\documentclass[final,12pt]{alt2025} 
\title[Minimax-optimal and Locally-adaptive Online Nonparametric Regression]{Minimax-optimal and Locally-adaptive \\ Online Nonparametric Regression}
\usepackage{times}

\defcitealias{kuzborskij2020locally}{KCB20}	
\defcitealias{hazan2007online}{HM07}
\defcitealias{gaillard2015chaining}{GG15}
\defcitealias{cesa2017algorithmic}{CB3G17}

\altauthor{%
 \Name{Paul Liautaud} \Email{paul.liautaud@sorbonne-universite.fr}\\
 \addr Sorbonne Université, CNRS, LPSM, F-75005 Paris, France
 \AND
 \Name{Pierre Gaillard} \Email{pierre.gaillard@inria.fr}\\
 \addr Université Grenoble Alpes, INRIA, CNRS, Grenoble INP, LJK, Grenoble, 38000, France%
 \AND
 \Name{Olivier Wintenberger} \Email{olivier.wintenberger@sorbonne-universite.fr}\\
 \addr Sorbonne Université, CNRS, LPSM, F-75005 Paris, France \\
 Institut Pauli CNRS, Vienna University, Oskar Morgenstern Platz 1, 1090 Wien, Austria%
}

\usepackage{tikz}
\usepackage{pgfplots}
\usepackage{pgfplotstable}
\pgfplotsset{compat=1.7}
\usepgfplotslibrary{groupplots}

\usepackage[utf8]{inputenc} 
\usepackage[T1]{fontenc}    
\usepackage{hyperref}       
\usepackage{url}            
\usepackage{booktabs}       
\usepackage{amsfonts}       
\usepackage{amsmath} 
\usepackage{mathrsfs}       
\usepackage{pifont} 
\usepackage{dsfont}
\usepackage{nicefrac}       
\usepackage{microtype}      
\usepackage{xcolor}         
\usepackage{amssymb}
\usepackage{comment}

\usepackage{bm}
\usepackage{caption}
\usepackage{graphicx}
\usepackage{tikz}
\usepackage{float}
\usepackage{enumitem}
\usepackage{wrapfig}
\usepackage{array}

\usepackage{multirow}

\DeclareMathSizes{10}{9}{6}{5}

\newcommand{\C}{\mathscr C}
\newcommand{\R}{\mathbb{R}} 
\newcommand{\X}{\mathcal{X}} 
\newcommand{\ind}[1]{\mathds{1}_{ #1 }}
\newcommand{\sumT}{\sum_{t=1}^T} 
\newcommand{\F}{\mathcal{F}} 
\newcommand{\N}{\mathcal{N}} 
\newcommand{\Lcal}{\mathcal{L}} 
\newcommand{\T}{\mathcal{T}} 
\newcommand{\Pcal}{\mathcal{P}} 
\newcommand{\p}{\mathrm{p}} 
\newcommand{\depth}{\mathrm{d}} 
\renewcommand{\root}{\mathrm{root}}
\newcommand{\Reg}{\operatorname{Reg}}

\renewcommand{\p}{\mathbf{p}}

\newcommand{\w}{\mathbf{w}}
\newcommand{\g}{\mathbf{g}}
\newcommand{\W}{\mathcal{W}}
\newcommand{\argmin}{\operatornamewithlimits{arg \, min}}
\renewcommand{\L}{\mathcal{L}}
\renewcommand{\p}{\mathrm{p}} 

\renewcommand{\root}{\mathrm{root}}
\renewcommand{\le}{\leqslant}
\renewcommand{\leq}{\leqslant}
\renewcommand{\ge}{\geqslant}
\renewcommand{\geq}{\geqslant}

\usepackage{tikz}
\usetikzlibrary{trees}
\usepackage{forest}

\newcommand{\nosemic}{\renewcommand{\@endalgocfline}{\relax}}
\makeatother
\SetKwComment{Comment}{\# }{}

\SetCommentSty{mycommfont}

\newtheorem{cor}{Corollary}
\newtheorem{defi}{Definition}
\newtheorem{assum}{Assumption}

\begin{document}

\maketitle

\begin{abstract}
We study adversarial online nonparametric regression with general convex losses and propose a parameter-free learning algorithm that achieves minimax optimal rates. Our approach leverages chaining trees to compete against Hölder functions and establishes optimal regret bounds. While competing with nonparametric function classes can be challenging, they often exhibit local patterns - such as local Hölder continuity - that online algorithms can exploit. Without prior knowledge, our method dynamically tracks and adapts to different Hölder profiles by pruning a core chaining tree structure, aligning itself with local smoothness variations. This leads to the first computationally efficient algorithm with locally adaptive optimal rates for online regression in an adversarial setting. Finally, we discuss how these notions could be extended to a boosting framework, offering promising directions for future research.
\end{abstract}


\section{Introduction}

Observing a stream of data \(x_1, x_2, \dots\), an online regression algorithm sequentially predicts a function \(\hat f_t\in \mathbb{R}^{\mathcal{X}}\) at each time step \(t \geq 1\) based on the current input \(x_t \in \mathcal{X} \subset \mathbb{R}^d\), where \(d \geq 1\). The accuracy of these predictions is measured using a sequence of convex loss functions \((\ell_t)_{t \geq 1}\). Examples include the absolute loss \(\ell_t(\hat y) = |\hat y - y_t|\) and the squared loss \((\hat y - y_t)^2\), for some observation \(y_t \in \R\).  
The performance of an online regression algorithm is evaluated through its \emph{regret} relative to a competitive class of functions \(\mathcal{F} \subset \mathbb{R}^{\mathcal{X}}\), defined over a time horizon \(T \geq 1\) as  

\begin{equation}
    \label{eq:reg}
    \Reg_T(f) := \sum_{t=1}^{T} \ell_t(\hat f_t(x_t)) - \sum_{t=1}^{T} \ell_t(f(x_t))\,, \qquad \forall f \in \F.
\end{equation}

The function class \(\mathcal{F}\) is typically chosen to capture smooth or structured relationships in the data, such as Lipschitz functions, which are commonly used to model nonparametric regression problems.  

Unlike traditional batch regression methods, which train models on the full dataset \(\{(x_s, \ell_s)\}_{s=1}^{T}\), online regression algorithms - see \citet{cesa2006prediction} for a reference textbook - make predictions sequentially, updating \(\hat f_t\) at each step using only past observations \(\{(x_s, \ell_s)\}_{s=1}^{t-1}\). This sequential and adaptive learning paradigm allows algorithms to capture complex and evolving patterns in the data without requiring strong assumptions, such as i.i.d. observations.

A fundamental goal in online regression is to design algorithms that are \emph{minimax-optimal} in an adversarial setting, meaning they achieve the best possible regret guarantees over the worst-case data sequence - see \citet{rakhlin2014online, rakhlin2015online}. Existing methods, such as those of \citet{gaillard2015chaining, cesa2017algorithmic}, can attain minimax rates when the regularity of the functions in the competitive class is known beforehand, but they do not extend to cases where the function's smoothness varies across the domain, requiring a more flexible and adaptive strategy. Later, \citet{kuzborskij2020locally} developed a locally adaptive algorithm, but it achieves a suboptimal regret rate. Thus, designing algorithms that adapt locally to unknown regularities and variations while maintaining optimal regret guarantees remains a key open challenge.

In this work, we propose a computationally efficient online learning algorithm that achieves \emph{locally adaptive minimax regret} without requiring prior knowledge of the competitor’s regularity. Our method builds on chaining trees and leverages an adaptive pruning mechanism that dynamically adjusts to local smoothness variations in the competitor function. Inspired by prior work on tree-based online learning \citep{kuzborskij2020locally}, we introduce a core tree structure that selects prunings in an optimal way, ensuring adaptivity to different Hölder profiles. This leads to the first online regression algorithm that is both minimax-optimal and locally adaptive, bridging the gap between computational efficiency, minimax rates, and local adaptivity. Additionally, our algorithm is general and applies to both convex and exp-concave loss functions, achieving optimal regret guarantees under mild assumptions. Finally, we validate our theoretical results with numerical experiments\footnote{The code to reproduce all numerical experiments can be found \href{https://github.com/paulliautaud/Minimax-Locally-Adaptive-Online-Regression}{here}.} demonstrating the practical benefits of our approach.

As a conclusion and perspective, we highlight how our approach shares similarities with boosting’s iterative refinement process and discuss how this connection could inspire future work in online regression.

\subsection{Related work}

\subsubsection{Online nonparametric regression}  

\citet{vovk2006metric} introduced online nonparametric regression with general function classes. \citet{cesa2006prediction} developed an algorithm that exploits loss functions with good curvature properties, such as exp-concavity, to achieve fast regret rates in adversarial settings. \citet{rakhlin2014online} further advanced the minimax theory, providing a non-polynomial algorithm that is optimal for regret in cumulative squared errors of prediction. This theory was later extended to general convex losses in \citet{rakhlin2015online}.  
A significant step toward computational efficiency was made by \citet{gaillard2015chaining, cesa2017algorithmic} designed a polynomial-time chaining algorithm that achieves minimax regret when the regularity of the competitor is known. They also observed that the same algorithm, with a different tuning, remains minimax-optimal for general convex losses.  

In the batch statistical setting with i.i.d. data, the convergence rates of tree-based aggregation methods have been primarily studied in the context of random forests; see \citet{biau2016guided} for a survey. Avoiding early stopping and overfitting, the purely random forests of \citet{arlot2014analysis} achieve minimax rates for i.i.d. nonparametric regression. Closer to our setting, \citet{mourtada2017universal} studied the aggregation of Mondrian trees trained sequentially but in a batch (i.i.d.) statistical framework. While their method adapts to the regularity of the unknown regression function in well-specified settings, it does not extend to adversarial environments.

\subsubsection{Regret against $\alpha$-Hölder competitors and local adaptivity}

Considering $\mathcal F$ as the set of Lipschitz functions ($\alpha = 1$ in Equation \eqref{eq:holder_func}) for any constant $L > 0$, \citet{hazan2007online} introduced the corresponding minimax regret. They proved that for $d=1$, the minimax rate is $O(\sqrt{L T})$ for any convex loss, motivating the design of an algorithm that localizes at an optimal rate depending on $L$. The knowledge of $L$ is crucial for their procedures to prevent regret from growing linearly with $L$.  
Going one step further, \citet{kuzborskij2020locally} demonstrated the adaptability of tree-based online algorithms by introducing an oracle pruning procedure in the regret analysis, given a core tree. Tracking the best pruning goes back to \citet{helmbold1995predictingpruning} and \citet{margineantu1997pruning}. \citet{kpotufe2013regression} also designed adaptive pruning algorithms based on trees to partition the instance space $\mathcal{X}$ optimally and sequentially. Competing with an oracle pruning in nonparametric regression enables adaptation to the local regularities $(L,\alpha)$ of $\alpha$-Hölder continuous functions - see \eqref{eq:holder_func}. Indeed, the implicit multi-resolution nature of pruning allows the depth of the leaves to align with local Hölder constants: the larger the constant, the deeper the pruning, as finer partitions are needed to capture variations in the function.  

However, existing methods either require prior knowledge of local Hölder constants in and the exponent rate $\alpha$, or they fail to attain minimax regret rates in polynomial time. The problem of designing a computationally efficient, minimax-optimal algorithm that adapts to local regularities remained open. One solution to this problem relies on an adaptive pruning approach on chaining trees, dynamically tracking and aligning with different Hölder profiles to adjust the depth of partitioning in an online manner. We propose an algorithm that efficiently adapts to local smoothness variations (both in $L$ and $\alpha$) without requiring prior knowledge of the underlying function regularities. Table~\ref{tab:comparison} presents an overview of our main result alongside previous advancements in the field of online (and local) nonparametric regression.

\begin{table}
\centering
\begin{tabular}{clc}
\toprule
\textbf{References} & \textbf{Assumptions} & \textbf{Upper bound} \\
\midrule
\multirow{2}{*}{This paper} 
    & $(\ell_t)$ exp-concave, $L > 0$ unknown & $\textstyle \min\big\{\sqrt{LT},L^{\frac{2}{3}}T^{\frac{1}{3}}\big\}$ \\
    & $(\ell_t)$ convex, $L > 0$ unknown & $\sqrt{LT}$ \\
\midrule
\citetalias{kuzborskij2020locally} 
    & $(\ell_t)$ square loss, $L > 0$ unknown & $\sqrt{LT}$ \\
\midrule
\multirow{2}{*}{\citetalias{hazan2007online}} 
    & $(\ell_t)$ absolute loss, $L > 0$ known & $L^{\frac{1}{3}}T^{\frac{2}{3}}$ \\
    & $(\ell_t)$ square loss, $L > 0$ known & $\sqrt{LT}$ \\
\midrule
\citetalias{gaillard2015chaining} 
    & $(\ell_t)$ square loss, $L = 1$ known & $T^{\frac{1}{3}}$ \\
\midrule
\citetalias{cesa2017algorithmic} 
    & $(\ell_t)$ convex, $L = 1$ known & $\sqrt{T}$ \\
\bottomrule
\end{tabular}
\caption{Comparison of regret guarantees for recent algorithms in online nonparametric regression with dimension $d = 1$ and smoothness rate $\alpha = 1$.}
\label{tab:comparison}
\end{table}



\subsection{Contributions and outline of the paper}  

We first present, in Section \ref{section:minimax_CT}, a parameter-free online learning method that leverages a chaining tree structure and achieves minimax regret over $\alpha$-Hölder continuous functions with global exponent rate $\alpha \leq 1$. Next, in Section \ref{section:locally_adaptive_regression}, we introduce a core tree adaptive algorithm that dynamically tracks and adjusts to local smoothness variations through an adaptive pruning mechanism, enabling it to efficiently compete against functions with different local regularities. We prove that our approach achieves an optimal locally adaptive regret bound in an adversarial setting. In particular, we show that our algorithm adapts to the curvature of the loss functions and remains optimal for both general convex and exp-concave losses.

Finally, we include numerical experiments in the supplementary materials (Appendix \ref{appendix:xp}) to illustrate our results on a synthetic dataset. As interesting perspectives, we also draw connections between our approach and boosting techniques, suggesting a potential foundation for a boosting theory in adversarial online regression.  

\section{Minimax regret with chaining trees: a parameter-free online approach}
\label{section:minimax_CT}

\paragraph{Setting and notations.}

We consider that data $x_1,x_2,\dots \in \X$ arrive in a stream. At each time step \(t \geq 1\), the algorithm updates $\hat f_t$, receives \(x_t \in \mathcal X\) and predicts \(\hat f_t(x_t) \in \R\).
Then, a loss function \(\ell_t\ : \R \to \R\) is disclosed. The learner incurs loss \(\ell_t(\hat f_t(x_t))\) and considers gradients to update strategies for time \(t+1\).
We assume that  $(\ell_t)$ are convex, $G$-Lipschitz and attain one minimum within \([-B, B]\), for some \(B > 0\). The input space \(\X\) is a bounded subspace of \(\R^d\), \(d \geq 1\). We write $|\X'| = \sup_{x,x'\in \X'} \|x-x'\|_\infty < \infty$ for any $\X' \subset \X$ and \([N] = \{1,\dots,N\}\) for \(N \ge 1\). 
\par
In this section, we present our first contribution: an online learning algorithm (Algorithm \ref{alg:training_CT}) that leverages a specialized decision tree structure, referred to as chaining trees, which we introduce in the next section. Specifically, we establish in Theorem \ref{theorem:minimax_CT_param_free} that our procedure achieves minimax-optimal regret in nonparametric regression over the class of Hölder-continuous functions.

\subsection{Chaining tree} 

Tree-based methods are conceptually simple yet powerful - see \citet{breiman2017classification}. They consist in partitioning the feature space into small regions and then fitting a simple model in each one.
Given \(\X \subset \R^d\), a \emph{regular decision tree} \((\T, \bar \X, \bar \W ) \) is composed of the following components: \label{def:regular_decision_tree}
\vspace{-0.4cm}
\begin{wrapfigure}{r}{0.3\textwidth}
  \begin{center} \vspace{-0.4cm}
 \begin{tikzpicture}[every node/.style={scale=0.8},scale=0.35]
  \node (p1) at (0,1) {};
  \node (p2) at (3,1) {} ;
  \node (p3) at (3,2) {};
  \node (p4) at (10,2) {};
  \node (c1) at (0,8) {};
  \node (c2) at (5,8) {};
  \node (c3) at (5,7) {};
  \node (c4) at (10,7) {};
    \draw[black, line width=1pt] (0,10)--(10,10) node[midway,above]{\(\theta_1\)};
   \draw[blue, line width=1pt] (0,4)--(2.5,4) node[midway,above]{$\theta_4$};
    \draw[blue, line width=1pt] (2.5,5)--(5,5) node[midway,above]{$\theta_5$};
    \draw[blue,line width=1pt] (2.5,4)--(2.5,5) ;
   \draw[red,line width=1pt] (c1)--(c2) node[midway,above]{\(\theta_2\)};
   \draw[red,line width=1pt] (0,8)--(5,8)--(5,7)--(10,7) ;
   \draw[red,line width=1pt] (c3)--(c4) node[midway,above]{\(\theta_3\)};
    \draw[blue, line width=1pt] (5,4.5)--(7.5,4.5) node[midway,above]{$\theta_6$};
    \draw[blue, line width=1pt] (7.5,4)--(10,4) node[midway,above]{$\theta_7$};
    \draw[blue,line width=1pt] (7.5,4.5)--(7.5,4) ;
    \draw[blue,line width=1pt] (5,4.5)--(5,5) ;
    \node () at (10.5,10.1) {\(\X\)};
    \fill (3.15,10) circle (5pt);
    \node at (3.15,10.8) {$x$};
    \fill (3.15,8) circle (5pt);
    \fill (3.15,5) circle (5pt);
    \draw[dash pattern=on 2pt off 2pt, thick] (3.15,10)--(3.15,5) ;
    \draw[->,>=latex,thick] (4,2.6) to[in=-45, out=90] (3.2,4.9) ;
    \node[below] at (5,2.6) {$\texttt{path}_\T(x)=\{1,2,5\}$};
    \node[below] at (5,1.5) {$\hat f(x) = \theta_1 + \theta_2 + \theta_5$};
  \end{tikzpicture}
  \end{center}
  \vspace{-0.55cm}
  \caption{Example of a CT over $\X \subset \R$.}
\label{fig:schematic_CT}
\end{wrapfigure}
\begin{itemize}[nosep,topsep=-\parskip]
    \item a finite rooted ordered regular tree $\T$ of degree $\deg(\T)$, with nodes $\N(\T)$ and leaves or terminal nodes \(\Lcal(\T) \subset \N(\T)\).  The root and depth of $\T$ are respectively denoted by $\mathrm{root}(\T)$ and $\depth(\T)$. Each interior node $n \in \N(\T) \backslash \Lcal(\T)$ has $\deg(\T)$ children. The parent of a node $n$ is referred to as $\p(n)$ and its depth as $\depth(n)$;
    \item a family of sub-regions \(\bar \X = \{\X_n, \, n \in \N(\T)\}\) consisting of subsets of $\X$ such that for any interior node $n$, $\{\X_m: \p(m) = n\}$ forms a  partition of $\X_n$; 
    \item a family of prediction functions \(\bar \W = \{h_n:\X \to \R, n \in \N(\T)\}\) associated to each node such that $h_n(x) = 0$ for all $x \notin \X_n$.
\end{itemize}

The standard method of \citet{breiman2017classification} for predicting with a decision tree is to use the partition induced by the leaves $\sum_{n \in \L(\T)} h_n(x)$, $x \in \X$. On the contrary, the chaining tree that we define below, preforms multi-scale predictions by combining the predictions from all nodes.

\begin{defi}[Chaining-Tree]
    \label{def:chaining_tree}
    A Chaining-Tree (CT)  prediction function $\hat f$ is defined as 
    \[
        \hat f(x) = \sum_{n \in \N(\T)} h_n(x) \,, \qquad x \in \X\subset \R^d\,,
    \]
    where the regular prediction tree \((\T, \bar \X, \bar \W) \) satisfies:
    \begin{itemize}[nosep, topsep=-\parskip]
    \item the prediction functions $h_n$ are constant $h_n(x) = \theta_n \ind{x \in \X_n}$,  $\theta_n \in \R$. We denote them by $\theta_n$ by abuse of notation;
    \item the degree  $\deg(\T) = 2^d$ and for any interior node $n$, $\{\X_m: \p(m) = n\}$ forms a regular partition of $\X_n$ in infinite norm. In particular, this implies $|\X_m| = |\X_{\p(m)}|/2$.
\end{itemize} 
\end{defi}

We provide a schematic illustration in Figure \ref{fig:schematic_CT}. Chaining trees are closely related to the chaining technique introduced by \citet{dudley1967sizes}, which is at the core of algorithms addressing function approximation tasks. This method involves a sequential refinement process, that is - roughly speaking - growing a sequence of refining approximations over a function space. It was first introduced to design concrete online learning algorithm with optimal rates by \citet{gaillard2015chaining}. 

\subsection{First algorithm: the online training of a chaining-tree}

We introduce in this section an explicit Algorithm \ref{alg:training_CT} to sequentially train our CT $\T$ over time.

\begin{algorithm2e}[H]
\caption{Training CT $\T$ at time $t \geq 1$}\label{alg:training_CT}
\SetKwInOut{Input}{Input}
\SetKwInOut{Output}{Output}
\Input{$(\theta_{n,t})_{n\in \N(\T)}$ (node predictors of $\T$), $(g_{n,t})_{n \in \N(\T)}$ (gradients - later specified).}
\For{$n \in \N(\T)$}{
Predict $\hat f_t(x_t) =\sum_{n\in \N(\T)} \theta_{n,t}\ind{x_t \in \X_n}$\;
Find $\theta_{n,t+1} \in \R$ to approximately minimize 
\begin{equation}
\label{eq:grad-step}
 \theta_n \mapsto \ell_t(\hat f_{-n,t}(x_t) + \theta_n\ind{x_t \in \X_n}) \quad \text{with} \quad \hat f_{-n,t}(x_t) = \hat f_t(x_t) - \theta_{n,t}\ind{x_t \in \X_n}
\end{equation}
using gradient $g_{n,t} = \left[\frac{\partial \ell_t\big(\hat f_{-n,t}(x_t) + \theta_n\ind{x_t \in \X_n}\big)}{\partial \theta_n}\right]_{\theta_n = \theta_{n,t}}$.
}
\Output{$(\theta_{n,t+1})_{n\in \N(\T)}$}
\end{algorithm2e}
To keep things concise, the gradient minimization step in \eqref{eq:grad-step} is expressed as:  \begin{equation}
    \label{eq:grad-step_CT}
    \theta_{n,t+1} \gets \texttt{grad-step}(\theta_{n,t},g_{n,t})\,.
    \end{equation}
where the function $\texttt{grad-step}(\theta,g)$ stands for \emph{any} rule that updates $\theta \in \R$ from time $t$ to $t+1$ using some gradient $g \in \R$. 

\paragraph{Computation of the gradients.}  
At each time $t \geq 1$ and for each node $n \in \N(\T)$, the subgradient $g_{n,t}$ of the last loss $\ell_t(\hat f_t(x_t))$ with respect to $\theta_{n,t}$ can be computed explicitly using the chain rule:  

\begin{equation}
\label{eq:gradient_CT}
    g_{n,t} = \left[\frac{\partial \ell_t(\hat f_{-n,t}(x_t)+\theta\ind{x_t \in \X_n})}{\partial \theta}\right]_{\theta = \theta_{n,t}} = \ell_t'\big(\hat f_t(x_t)\big) \ind{x_t \in \X_n} \,,\qquad n \in \N(\T)\,,
\end{equation}  

which simplifies the computation of subgradients, as the dependence on $n$ only involves the indicator function. More precisely, the subroutine \texttt{grad-step}, detailed below, does not perform any update (i.e., $\theta_{n,t+1} = \theta_{n,t}$) when the gradient is zero (i.e., $x_t \notin \X_n$). All nonzero updates use the same subgradient $g_t = \ell_t'\big(\hat f_t(x_t)\big)$, which is based on the derivative of the loss of the strong learner's prediction.  

\paragraph{Online gradient optimization subroutine.} 
We now detail the subroutine \texttt{grad-step}, which, in our analysis, can be any online optimization algorithm satisfying the following regret upper-bound.

\begin{assum}
\label{assumption:parameter_free}
   Let $g_{n,1},\dots, g_{n,T} \in [-G,G]$ for $T \geq 1$, $G >0$, and $n \in \N(\T)$. We assume that the  parameters $\theta_{n,t}$ starting at $\theta_{n,1} \in \R$ and following the update~\eqref{eq:grad-step_CT} satisfy the linear regret bound: 
   \[
    \sum_{t=1}^T g_{n,t}(\theta_{n,t} - \theta_n) \le |\theta_n - \theta_{n,1}|\Big(C_1 \textstyle \sqrt{\sum_{t=1}^T|g_{n,t}|^2} +  C_2G\Big) \,,
    \]
    for some $C_1, C_2 >0$ and every $\theta_n \in \R$.
\end{assum} 

Such an assumption is satisfied by so-called \emph{parameter-free} online convex optimization algorithms, such as those described in \citet{cutkosky2018black, mhammedi2020lipschitz, orabona2016coin}. Specifically, by considering only the time steps where the gradients are nonzero, $T_n = \{1\leq t\leq T : g_{n,t} \neq 0 \}$, their procedure entails $O(G|\theta_n|\sqrt{|T_n|})$. Note that the constants $C_1,C_2$ often hide logarithmic factors in $T,G$ or $|\theta_n|$. These algorithms require no parameter tuning (though some need prior knowledge of $G$) and provide a regret upper bound that automatically scales with the parameter norm $|\theta_n|$. 
This property is crucial in analyzing our CT, where each node is tasked with correcting the errors of its ancestors in a more refined subregion of the input space. This multi-resolution aspect of the predictions leads us to consider $\theta_n$ that approach zero as $\depth(n)$ increases.

\paragraph{First result.} In the theorem below, we show that when resorting to such a subroutine into Algorithm~\ref{alg:training_CT}, our results are minimax-optimal with respect to $\C^\alpha(\X,L)$ the class of $\alpha$-Hölder continuous functions over $\X$ defined with $L>0$ and $\alpha \in (0,1]$ by
  \begin{equation}
  \label{eq:holder_func}
    \textstyle \C^\alpha(L,\X) := \big\{f:\X \to \R: |f(x) - f(x')| \leq L \|x - x'\|_\infty^\alpha\, , \; x,x'\in \X  \text{ and } \sup_{x \in \X} |f(x)| \leq B \big\}\, ,
  \end{equation}
with $B > 0$ such that $\ell_t$ has minimum lying in $[-B,B]$. We will refer to $L$ as the \emph{Hölder constant} and $\alpha$ to as the \emph{smoothness rate} or \emph{exponent}.

\begin{theorem}
    Let $T \geq 1, (\T, \bar \X, \bar \W_1)$ be a  CT with $\X_{\root(\T)} = \X$, $\theta_{n,1} = 0$ for all $n \in \N(\T)$ and $\depth(\T) = \frac{1}{d} \log_2 T$. Then, Algorithm~\ref{alg:training_CT} applied with a \emph{\texttt{grad-step}} procedure satisfying Assumption~\ref{assumption:parameter_free} achieves the regret upper bound
    \[
      \sup_{f\in \C^\alpha(\X,L)} \Reg_T(f) \leq GB(C_1\sqrt{T}+C_2) + GL|\X|^\alpha 
        \begin{cases}             
            \big(  \Phi(\frac{d}{2}-\alpha) C_1 + 4 C_2 + 1 \big) \sqrt{T} & \text{if $d< 2\alpha$}\,, \\ 
            \big(\frac{C_1}{d} \log_2 T + 4 C_2 + 1 \big) \sqrt{T} & \text{if $d=2\alpha$} \,, \\
            \big( \Phi(\frac{d}{2}-\alpha) C_1 + 4 C_2 + 1 \big) T^{1-\frac{\alpha}{d}} & \text{if $d > 2\alpha$} \,, 
    \end{cases}
    \]
    for any $L >0$ and $\alpha \in (0,1]$, where $\Phi(u) = |2^{u} - 1|^{-1}$.
    \label{theorem:minimax_CT_param_free}
\end{theorem}

The proof of Theorem \ref{theorem:minimax_CT_param_free} is postponed to Appendix \ref{appendix:proof_theorem:minimax_CT_param_free}.

\paragraph{Minimax optimality and adaptivity to $L$ and $\alpha$.} Note that the above rates are minimax optimal for online nonparametric regression with convex losses over $\C^\alpha(\X,L)$, as shown by \citet{rakhlin2015online} that provides a non-constructive minimax analysis for this problem (see also \citet{rakhlin2014online}). For the case of low-dimensional settings, where \( d \leq 2\alpha \), our bound is in \( O((B+L)\sqrt{T}) \). However, it has been demonstrated in \cite{rakhlin2015online} that faster rates \( O(T^{\frac{1}{3}}) \) can be attained when dealing with exp-concave losses. In the next section, we will address this by making our algorithm adaptive to the curvature of the loss functions.
A similar chaining technique was applied by \citet{gaillard2015chaining} to design an algorithm with minimax rates for the square loss or \citet{cesa2017algorithmic} in the partial information setting. However, unlike these works, our Algorithm \ref{alg:training_CT} does not require prior knowledge of neither $L$ nor $\alpha$ and automatically adapts to them. This is achieved through the use of \emph{parameter-free} subroutines that satisfy Assumption~\ref{assumption:parameter_free} and automatically adapt to the norm of $\theta_n$. 

\paragraph{Comparison to standard adaptive OCO methods in $\R^{|\N(\T)|}$.}
A key point of Algorithm \ref{alg:training_CT} is its \emph{node-specific} descent, which differs from standard adaptive OCO optimizing a global parameter. 
For each node $n \in \N(\T)$ we obtain a regret upper-bound of $O(|\theta_n|\sqrt{\sum_t |g_{n,t}|^2})$ yielding an overall regret in $O(\sum_{n}|\theta_n|\sqrt{\sum_t |g_{n,t}|^2})$, with $g_{n,t}$ defined as in \eqref{eq:gradient_CT}. Notably, thanks to the structure of the chaining-tree, $g_{n,t} = 0$ when the data $x_t$ does not fall in the corresponding sub-region of node $n$ and this leads to an overall regret scaling as $O(G\sum_n |\theta_n|\sqrt{|T_n|})$ where $T_n$ is the set of time steps for which $g_{n,t} \neq 0$.

One may wonder whether Algorithm \ref{alg:training_CT} could be
reduced to an adaptive Online Mirror Descent (OMD) on a \emph{global} parameter $\boldsymbol{\theta} = (\theta_n)_{n \in \N(\T)} \in \R^{|\N(\T)|}$. This would result in an estimation regret bound, for any $p,q \geq 1$ such that $\frac{1}{p} + \frac{1}{q} = 1$, \[\textstyle O\Big(\|\boldsymbol{\theta}\|_p\sqrt{\sum_{t=1}^T \|\g_t\|_q^2}\Big) \quad \text{where} \quad \g_t = \nabla_{\boldsymbol{\theta}} \ell_t\big(\sum_{n \in \N(\T)} \theta_{n,t}\ind{x_t \in \X_n}\big) = (g_{n,t})_{n \in \N(\T)}\,,\] with $g_{n,t}$ as in \eqref{eq:gradient_CT}. Moreover, we have for $q \geq 2$
\begin{align}
\textstyle	\sum_n |\theta_n| \sqrt{\sum_t |g_{n,t}|^2} 
	& 	\leq \textstyle \|\boldsymbol \theta\|_p \Big(\sum_n \big(\sum_t |g_{n,t}|^2\big)^{\frac{q}{2}}\Big)^{\frac{1}{q}}   &\leftarrow \text{by H\"older's inequality}  \notag \\
	& 	\textstyle \leq \|\boldsymbol \theta\|_p \Big(\sum_t \big(\sum_n |g_{n,t}|^q\big)^{\frac{2}{q}}\Big)^{\frac{1}{2}} & \leftarrow \text{by Minkowski's inequality with $\frac{q}{2} \geq 1$} \notag \\
	&= \textstyle \|\boldsymbol\theta\|_p \sqrt{\sum_t \|\g_t\|_q^{2}}\,. \label{eq:comparison_OMD}
\end{align}

Remarkably, \eqref{eq:comparison_OMD} shows that our Algorithm~\ref{alg:training_CT} consistently achieves a lower regret compared to any global adaptive OMD subroutine for $q \geq 2$ - including adaptive version of OGD ($p=q=2$) and of EG ($p=1, q=\infty$).

Finally, in our analysis in Appendix \ref{appendix:proof_theorem:minimax_CT_param_free}, Proof of Theorem \ref{theorem:minimax_CT_param_free}, such an adaptive OCO method would result in an overall estimation regret of $O(\sqrt{T} \sum_{n} |\theta_n|)$. By grouping by the level of the CT $\T$ (see Equation \eqref{eq:estimation_error}), with $|\theta_n| \propto 2^{-\alpha m}, m \in [\depth(\T)]$, we get a regret of $O(2^{-\alpha m}|\{n:\depth(n) = m\}| \sqrt{T})$ for each level instead of 
$O(2^{-\alpha m}\sqrt{|\{n:\depth(n) = m\}| T})$, which is insufficient to recover the same minimax rates.

\paragraph{Complexity.} Although the formal definition of our algorithm requires constructing a decision tree with $|\N(\T)| = 2^{\depth(\T) d} = T$ nodes, it remains tractable, similar to the approach in \citet{gaillard2015chaining}. At each round, the input $x_t$ falls into one node per level of the tree constituting $\texttt{path}_\T(x_t) = \{n \in \N(\T): x_t \in \X_n\}$, since $\{\X_n, \depth(n) = m\}$ forms a partition of $\X$ for any depth $1\leq m \leq \depth(\T)$ - see Figure \ref{fig:schematic_CT} for schematic comprehension. Consequently, most subgradients in \eqref{eq:gradient_CT} are zero, and \texttt{grad-step} only needs to be called $\depth(\T) = \frac{1}{d} \log_2 T $ times per round, each using the same gradient $g_t$. Thus, the loop in Algorithm \ref{alg:training_CT} can be rewritten to explore only the nodes along $\texttt{path}_\T(x_t)$, significantly reducing computational complexity. The overall space complexity is at most $O(|\N(\T)|) = O(T)$. It can be improved noticing that nodes in the tree do not need to be created until at least one input falls into that node.

\paragraph{Unknown input space.} In practice our procedure can be easily extended to the case where $\X$ is unknown beforehand and is sequentially revealed through new inputs $x_t \in \R^d$ (similarly to \citet{kuzborskij2020locally}). This can be done either through a doubling trick (starting with $\X = [-1,1]^d$ and restarting the algorithm with an increased diameter by at least a factor of $2$ each time an input falls outside of the current tree) or by creating a new CT around $x_t$ 
that runs in parallel, whenever a new point $x_t$ falls outside the existing trees.

\section{Optimal and locally adaptive regret in online nonparametric regression}

\label{section:locally_adaptive_regression}

In the previous section, we demonstrated that our Algorithm \ref{alg:training_CT} achieves minimax regret $O(LT^{(d-\alpha)/d})$ compared to Hölder functions $\C^\alpha(\X,L)$. This bound scales linearly with the constant $L$ and raises the question of whether our approximation method could be adapted to fit subregions with lower variation. Our second contribution is an algorithm that adapts on the local Hölder profile of the competitor.  For any $f\in \C^\alpha(\X,L)$, $\alpha\in(0,1], L > 0$, and some subset $\X_n \subset \X$, the local Hölder constant $L_n(f)$ satisfies
\begin{equation}\label{eq:deflochold} L_n(f) \leq L \quad \text{and} \quad |f(x) - f(x')| \leq L_n(f) \|x - x'\|_\infty^\alpha \,,
\end{equation}
for every $x,x'\in \X_n$. Recall that we assume that for any $f \in \C^\alpha(\X,L), \sup_{x \in \X} |f(x)| \leq B$. We define $[\cdot]_B := \min(B,\max(-B,\cdot))$ the clipping operator in $[-B,B]$ and a uniform discretization grid $\Gamma$ with precision $\varepsilon = T^{-\frac{1}{2}}$ as the set of $K = \lceil 2B/ \varepsilon \rceil$ constants 
\vspace{-0.1cm}
\[\Gamma := \{\gamma_k = -B + (k-1)\varepsilon \,, k = 1,\dots, K\} \subset [-B,B].\]

\paragraph{Locally adaptive algorithm.} We base our predictions on a combination of several regular decision tree predictions (see Section \ref{def:regular_decision_tree}). The latter are sitting in nodes of a \emph{core tree} \((\T_0, \bar \X, \bar \W)\), with \(\bar \W = \{(\hat f_{n,k})_{k=1}^K, \, n \in \N(\T_0)\}\). 
In our main Algorithm \ref{alg:Local_Adapt_Algo}, referred to as \emph{Locally Adaptive Online Regression}, the core tree \(\T_0\) provides an average prediction at each time step \(t \geq 1\) as follows:
\[
\textstyle \hat f_t(x_t) = \sum_{n \in \N(\T_0)} \sum_{k=1}^K w_{n,k,t} \hat f_{n,k,t}(x_t)\,,
\]
where, for each pair \((n,k) \in \N(\T_0) \times [K]\) \begin{itemize}[nosep,topsep=-\parskip]
    \item \(\hat f_{n,k,\cdot}\) is a clipped predictor associated with a CT \(\T_{n,k}\) (see Definition \ref{def:chaining_tree}), rooted at \(\X_{\root(\T_{n,k})} = \X_n \in \bar \X\) and starting at \(\theta_{\root(\T_{n,k}),1} = \gamma_k \in \Gamma\), $\theta_{n',1} = 0$ for $n' \in \N(\T_{n,k}) \setminus \{\root(\T_{n,k})\}$;
    \item the weight \(w_{n,k,t}\) adjust the contribution of the predictor \(\hat f_{n,k,t}\) such that the sum of all weights over the tree satisfies \(\sum_{n \in \N(\T_0)} \sum_{k \in [K]} w_{n,k,t} = 1\) at any time \(t \geq 1\).
\end{itemize}


 First, Algorithm \ref{alg:Local_Adapt_Algo} sequentially trains the weights $(w_{n,k})_{(n,k) \in \N(\T_0) \times [K]}$ using two key subroutines: \texttt{weight} and \texttt{sleeping}, both inspired by classical expert aggregation methods. Specifically, the $\texttt{weight}(\tilde \w,\tilde \g)$ subroutine refers to any general algorithm updating weights $\tilde \w$ with a given gradient $\tilde \g$ and satisfies the following Assumption \ref{assumption:second_order_algo}.  

\begin{assum}
\label{assumption:second_order_algo}
Let \(\tilde \g_1, \dots, \tilde \g_T \in [-G,G]^{K\times|\N(\T_0)|}\), for \(T \ge 1\) and \(G > 0\). We assume that the weight vectors \(\tilde \w_t\), initialized with a uniform distribution \(\tilde \w_1\) and updated via \texttt{weight} in Algorithm \ref{alg:Local_Adapt_Algo}, satisfy the following linear regret bound:
\[
\textstyle
\sumT \tilde \g_t^\top \tilde \w_t  - \tilde g_{n,k,t} \leq C_3\sqrt{\log(K|\N(\T_0)|)\sumT \big(\tilde \g_t^\top \tilde \w_t - \tilde g_{n,k,t}\big)^2} + C_4G,
\]
for some constants \(C_3, C_4 > 0\) and for every \(n \in \N(\T_0)\), \(k \in [K]\).
\end{assum}

Well-established aggregation algorithms, such as those from \citet{gaillard2014second}, \citet{koolen2015second}, and \citet{wintenberger2017optimal}, exhibit such second-order linear regret bounds.  

\begin{algorithm2e}[h]
\caption{Locally Adaptive Online Regression}\label{alg:Local_Adapt_Algo}
\SetKwInOut{Input}{Input}
\SetKwInOut{Output}{Output}
\Input{A core regular tree \((\T_0, \bar \X, \bar \W)\) with root \(\X\), bounds $G,B > 0$.

Initial prediction functions $\hat f_{n,k,1}= \tilde f_{n,k,1} = \theta_{\root(\T_{n,k}),1}\ind{x \in \X_n}$ associated to CT $\T_{n,k}, k \in [K], n \in \N(\T_0)$.

Initial uniform weights \(\tilde \w_1 = (\tilde w_{n,k,1})_{n \in \N(\T_0), k \in [K]}\).}
\For{$t=1$ \KwTo $T$}{
Receive $x_t$\;
$\N_t \gets \texttt{path}_{\T_0}(x_t)$\; 
$\w_t \gets \texttt{sleeping}(\tilde \w_t,\N_t)$ \label{alg:line_sleeping}\;
Predict $\hat f_t(x_t) = \sum_{n \in \N_t} \sum_{k = 1}^K w_{n,k,t}\hat f_{n,k,t}(x_t)$ \label{alg:line_prediction} \;
\Comment{Update weights of $\T_0$}
Reveal gradient $\tilde \g_t = \nabla_{\tilde \w_t} \ell_t( \sum_{n \in \N_t} \sum_{k=1}^K \tilde w_{n,k,t} \hat f_{n,k,t}(x_t) + \sum_{n \notin \N_t} \sum_{k=1}^K \tilde w_{n,k,t} \hat f_t(x_t) )$ \;
Udpate $\tilde \w_{t+1} \gets \texttt{weight}(\tilde \w_t, \tilde \g_t)$ \label{alg:line_weight} \;
\For{\(n \in \N_t, k \in [K]\)}{
\Comment{Update CT $\T_{n,k}$}
Reveal gradient $g_{n,k,t} = \ell_t'(\tilde f_{n,k,t}(x_t))$ \;
Update $\tilde f_{n,k,t}$ associated to CT $\T_{n,k}$ using Algorithm \ref{alg:training_CT} with $g_{n,k,t}$  \label{alg:line_boosting} \; 
Clip local predictor as $\hat f_{n,k,t+1 }= \big[\tilde f_{n,k,t+1}\big]_B$ \label{alg:line_clipping}\;
}
}
\Output{$\hat f_{T+1} = \sum_{n,k} w_{n,k,T+1}\hat f_{n,k,T+1}$}
\end{algorithm2e}

Since $\T_0$ partitions the input space $\X$, only a subset $\N_t$ of the nodes in $\N(\T_0)$ contributes to predictions at each round \(t \geq 1\). The set of active nodes is determined by $\N_t \gets \texttt{path}_{\T_0}(x_t)$, which maps the data point $x_t$ to the active nodes $\{n \in \N(\T_0) : x_t \in \X_n\}$. This structure mirrors the \emph{sleeping experts} framework introduced by \citet{freund1997sleeping, gaillard2014second}, and we incorporate it as a \texttt{sleeping} subroutine in Algorithm \ref{alg:Local_Adapt_Algo}. 
The weights $\w_t$ are computed using the $\texttt{sleeping}(\tilde \w_t,\N_t)$ subroutine, defined as follows for all $k \in [K]$ and $n \in \N(\T_0)$:
\begin{equation}
\label{eq:weights_sleeping}
w_{n,k,t} = 
0 \quad \text{if } n \notin \N_t\,, \qquad 
w_{n,k,t} = \frac{\tilde w_{n,k,t}}{\sum_{n' \in \N_t} \sum_{k'=1}^K \tilde w_{n',k',t}}  \quad \text{otherwise.}
\end{equation}
This ensures that only the active nodes are contributing to the average prediction.

Second, our Algorithm \ref{alg:Local_Adapt_Algo} also employs Algorithm \ref{alg:training_CT} to independently train the CTs 
\vspace{-0.2cm}
\[\textstyle \{\T_{n,k} \, , (n,k) \in \N(\T_0)\times [K]\}\] that reside within \(\T_0\). For each \((n,k) \in \N(\T_0)\times [K]\), \(\T_{n,k}\) is initialized with \(\theta_{\root(\T_{n,k}),1} = \gamma_k\) and \(\theta_{n',1} = 0\) for all \(n' \in \N(\T_{n,k}) \setminus \{\root(\T_{n,k})\}\), and is then updated at each time $t \geq 1$ via Algorithm \ref{alg:training_CT} with a given gradient $g_{n,k,t}$. Then, the local predictors associated to $(\T_{n,k})$ are clipped in $[-B,B]$.

\paragraph{Pruning as local adaptivity.}

Pruning techniques are frequently employed in traditional statistical learning involving subtrees to reduce overfitting or simplify models. In this context, each pruned tree represents a localized profile corresponding to a partition of $\X$. Our Algorithm \ref{alg:Local_Adapt_Algo} strives to learn the oracle pruning strategy to compete effectively against any $\alpha$-Hölder continuous function.

\begin{defi}[Pruning]
   Let \((\T_0,\bar \X,\bar \W)\) be some regular tree with \(\bar \W = \{(\hat f_{n,k})_{k \in [K]}, n \in \N(\T_0)\}\). 
A \emph{pruning} or \emph{pruned regular decision tree} \((\T, \tilde \X, \tilde \W)\) consists in a subtree, i.e. \(\N(\T) \subset \N(\T_0)\), with root \(\X_{\mathrm{root}(\T)}=\X_{\mathrm{root}(\T_0)}\) and prediction functions \(\tilde \W = \{\hat f_{n,k_n}, n \in \N(\T), k_n \in [K]\} \subset \bar \W\). It predicts, at each time $t \geq 1$,
\vspace{-0.3cm}
   \[ \textstyle \hat f_{\T,t}(x)=\sum_{n \in \L(\T)} \hat f_{n,k_n,t}(x), \quad x \in \X\,.\]
\vspace{-0.2cm}
 We denote $\Pcal(\T_0)$ the set of all prunings of $\T_0$.
   \label{def:pruning}
\end{defi}

Note that a pruning is a decision tree whose predictions are induced by its leaves, contrary to the core tree \(\T_0\). In particular, a prediction made by a leaf of a pruning is inherited from the associated node in \(\T_0\) before pruning.
We provide some illustration in Figure \ref{fig:tree_and_pruning}.

\tikzstyle{sleep} = [circle, minimum width=6pt, draw=black, rounded corners=1mm, fill=none, inner sep=0pt, ultra thick]
\tikzstyle{nosleep} = [circle, minimum width=6pt, draw=black, fill=black, ultra thick, inner sep=0pt]
\tikzstyle{line} = [draw=none, ultra thick]
\tikzstyle{chain} = [minimum width=10pt, draw=none]

\begin{figure}[htbp]
    \floatconts{fig:tree_and_pruning}
    {\vspace{-0.4cm} \caption{Example of a core tree \(\T_0\) with depth \(\depth(\T_0)=3\), $d=1$, in Fig. \ref{fig:maintree}. We give 2 pruned tree instances \(\T_1\) for a given Lipschitz function $f_1$ in Fig. \ref{fig:prun1} and \(\T_2\) for a second profile $f_2$ in Fig. \ref{fig:prun2}. In Fig. \ref{fig:maintree} all nodes $\N(\T_0)$ are awaken and predictive while $\T_1$ in Fig. \ref{fig:prun1} (resp. $\T_2$ in Fig. \ref{fig:prun2}) predicts with $\hat f_{2,k_2}, \hat f_{3,k_3}$ sitting in its leaves $\L(\T_1)$ (resp. with $\hat f_{2,k_2}, \hat f_{6,k_6}, \hat f_{7,k_7}$ sitting in its leaves $\L(\T_2)$). \ding{55} represents a pruned node.}}
    {%
        \subfigure[\small Core tree $\T_0$]{
            \label{fig:maintree}
            \begin{minipage}[c]{0.33\textwidth}
            \centering
                \begin{tikzpicture}[scale=0.92,
                level distance=1.6cm,
                level 1/.style={sibling distance=3cm},
                level 2/.style={sibling distance=1.5cm},
                every node/.style = {draw, line width=0.2cm, thick},
                edge from parent path={(\tikzparentnode.south) -- (\tikzchildnode.north)}]
                \begin{scriptsize}
                \node[chain] {$\{\hat f_{1,k}\}_{k=1}^K$}
                child { node[chain] {$\{\hat f_{2,k}\}_{k=1}^K$}
                    child { node[chain] {$\{\hat f_{4,k}\}_{k=1}^K$} 
                    }
                    child { node[chain] {$\{\hat f_{5,k}\}_{k=1}^K$} 
                    }
                }
                child { node[chain] {$\{\hat f_{3,k}\}_{k=1}^K$}
                    child { node[chain] {$\{\hat f_{6,k}\}_{k=1}^K$} 
                    }
                    child { node[chain] {$\{\hat f_{7,k}\}_{k=1}^K$} 
                    }
                };
                \end{scriptsize}
                \end{tikzpicture}
                \vspace{0.2cm}
            \end{minipage}
        }%
        \hfill
        \subfigure[\small \(\T_1 \in \Pcal(\T_0), f_1 \in \C^\alpha(\X,L)\)]{
            \label{fig:prun1}
            \begin{minipage}{0.3\textwidth}
               \centering \begin{tikzpicture}[scale=0.95,
                level distance=1.2cm,
                level 1/.style={sibling distance=2cm},
                level 2/.style={sibling distance=1cm},
                every node/.style = {draw, thick},
                edge from parent path={(\tikzparentnode.south) -- (\tikzchildnode.north)}]
                \begin{scriptsize}
                \node[chain] { $\cdot$ }
                child { node[chain] {$\hat f_{2,k_2}$}
                    child { node[chain] {\ding{55}} 
                    }
                    child { node[chain] {\ding{55}} 
                    }
                }
                child { node[chain] {$\hat f_{3,k_3}$}
                    child { node[chain] {\ding{55}} 
                    }
                    child { node[chain] {\ding{55}} 
                    }
                };
                \end{scriptsize}
                \end{tikzpicture} \\
                \vspace{-0.7cm} 
               \begin{tikzpicture}[scale=0.5]
    \begin{axis}[
        domain=-6:6, 
        samples=100, 
        axis x line=bottom,
        axis line style={-},
        xlabel=$\X$,
        ymin=0, ymax=1.2, 
        xtick=\empty, 
        axis y line=none,
        legend pos=north east,
        thick
    ]
    \addplot[
        thick,
        domain=-5:5
    ]
    {exp(-0.5*(x+3)^2) + 0.7*exp(-0.5*(x-3)^2)};
    \node at (axis cs:-2,0.6) [anchor=west] {$f_1$};
    \end{axis}
\end{tikzpicture}
            \end{minipage}
        }
        \hfill 
        \subfigure[\small \(\T_2 \in \Pcal(\T_0), f_2 \in \C^\alpha(\X,L)\)]{
            \label{fig:prun2}
            \begin{minipage}[c]{0.3\textwidth}
            \centering
                \begin{tikzpicture}[scale=0.95,
            level distance=1.2cm,
            level 1/.style={sibling distance=2cm},
            level 2/.style={sibling distance=1cm},
            every node/.style = {draw, thick},
            edge from parent path={(\tikzparentnode.south) -- (\tikzchildnode.north)}]
            \begin{scriptsize}
            \node[chain] { $\cdot$ }
            child { node[chain] {$\hat f_{2,k_2}$}
                child { node[chain] {\ding{55}} 
                }
                child { node[chain] {\ding{55}} 
                }
            }
            child { node[chain] { $\cdot$ }
                child { node[chain] {$\hat f_{6,k_6}$} 
                }
                child { node[chain] {$\hat f_{7,k_7}$} 
                }
            };
            \end{scriptsize}
            \end{tikzpicture} \\
                \vspace{-0.7cm} 
                \begin{tikzpicture}[scale=0.5]
                    \begin{axis}[
                        domain=-6:6,
                        samples=100,
                        axis x line=bottom,
                        axis line style={-},
                        xlabel=$\X$,
                        ymin=0, ymax=1.2,
                        xtick=\empty,
                        axis y line=none,
                        thick
                    ]
                    \addplot[
                        thick,
                        domain=-5:5
                    ]
                    {exp(-0.5*(x+3)^2) + 0.5*exp(-2*(x-1.5)^2) + 0.7*exp(-1*(x-4)^2)};
                    \node at (axis cs:-2,0.6) [anchor=west] {$f_2$};
                    \end{axis}
                \end{tikzpicture}
            \end{minipage}
        }
    }
\end{figure}

\paragraph{Complexity.}
Similar to before, even though our core tree \(\T_0\) involves at most \(O(|\N(\T_0)|) = O(\sqrt{T}2^{\depth(\T_0)d}) = O(T^{\frac 3 2})\) predictors after $T$ iterations, our algorithm remains computationally feasible, since at a time \(t\), only a subset of \(\depth(\T_0)\) nodes are active and updated with the \texttt{weight} subroutine. The resulting overall complexity is of order \(\frac{1}{d^2}\sqrt{T}\log_2(T)^2\) per step. 

\paragraph{Second result.}
In our main result (Theorem \ref{theorem:optimal_local_regret}), we prove that Algorithm \ref{alg:Local_Adapt_Algo} achieves a locally adaptive regret with respect to any $\alpha$-Hölder function. Indeed, we show an upper-bound regret that scales with the local regularities of the competitor. Meanwhile, we show that Algorithm \ref{alg:Local_Adapt_Algo} also adapts to the curvature of the losses: its regret performances improve when facing exp-concave losses (i.e., when $y \mapsto e^{-\eta \ell_t(y)}$ are concave for some $\eta>0$), as shown in the second part of Theorem \ref{theorem:optimal_local_regret}. 
Exp-concave losses include the squared, logistic or logarithmic losses. 
Note that for Assumption \ref{assumption:second_order_algo} to hold, the gradients \( \tilde \g_t \) must be bounded by \( G \) in the sup-norm. The Hölder assumption on \( f \) and the boundedness condition on \( \X \) alone are not sufficient. It is also essential that all predictions \( \hat f_{n,k,t}(x_t) \) are bounded, which is achieved through clipping in Algorithm \ref{alg:Local_Adapt_Algo} - see e.g., \citet{gaillard2015chaining, cutkosky2018black}. To simplify the presentation, we state the theorem here only for the case $d=1$ and $\alpha > 1/2$.

\begin{theorem}
    Let  \(\alpha \in (\frac 1 2,1], d=1, T \geq 1\) and $(\T_0, \bar \X, \bar \W)$ be a core regular tree with $\X_{\root(\T_0)} = \X$ and CT $\{\T_{n,k} : (n,k) \in \N(\T_0)\times [K]\}$ satisfying the same assumptions as in Theorem \ref{theorem:minimax_CT_param_free} and whose nodes root are initialized as $\theta_{\root(\T_{n,k}),1} = \gamma_k \in \Gamma$, for all $(n, k) \in \N(\T_0) \times [K]$.
    Then, Algorithm~\ref{alg:Local_Adapt_Algo} with a \emph{\texttt{weight}} subroutine as in Assumption~\ref{assumption:second_order_algo}, achieves the regret upper-bound with respect to any $f \in \C^\alpha(\X,L), L > 0$,
    \vspace{-0.2cm}
    \[
        \textstyle \Reg_T(f) \lesssim \inf_{\T \in \Pcal(\T_0)} \bigg\{ \sqrt{|\L(\T)|T} + |\L(\T)|
         + |\X|^\alpha  \sum_{n \in \L(\T)} L_n(f) 2^{-\alpha(\depth(n)-1)}
        \sqrt{|T_n|} \bigg\},
    \]
where $\lesssim$ is a rough inequality depending on $C_i, i=1,\dots,4, G$ and $L_n(f) \leq L, n \in \L(\T)$, are the local Hölder constants \eqref{eq:deflochold} of \(f\), and \(T_n = \{1 \leq t \leq T : x_t \in \X_n\}\).

Moreover, if \(\ell_1, \dots, \ell_T\) are exp-concave, one has:
    \[\textstyle \Reg_T(f) \lesssim \inf_{\T \in \Pcal(\T_0)} \left\{|\L(\T)| +  |\X|^\alpha \sum_{n \in \L(\T)} L_n(f) 2^{-\alpha(\depth(n)-1)}
        \sqrt{|T_n|}
    \right\}\]
where $\lesssim$ also depends on the exp-concavity constant.
\label{theorem:optimal_local_regret}
\end{theorem}

We state and prove a complete version of Theorem \ref{theorem:optimal_local_regret} in Appendix \ref{appendix:proof_theorem:optimal_local_regret}, for all \(\alpha \in (0,1], d\geq 1\).
As a remark, Algorithm \ref{alg:Local_Adapt_Algo} is not only adaptive to the local Hölderness of \(f\) (via \(L_n(f)\)), but also to the smoothness rate $\alpha \in (0,1]$. One could extend the previous results in Theorem \ref{theorem:optimal_local_regret} with some local smoothness $(\alpha_n)$ associated to the regularity of the function over the pruned leaves at the price of the interpretability of the bound in specific situations as below. 

\paragraph{Minimax optimality and adaptivity to the loss curvature.} Moreover,  Theorem~\ref{theorem:optimal_local_regret} yields the following corollary, which  demonstrates that our algorithm simultaneously achieves optimal rates for generic convex losses (i.e., similar rates to Theorem~\ref{theorem:minimax_CT_param_free}) and for exp-concave losses, while also adapting locally to the Hölder profile of the competitor - i.e. exhibiting dependencies to constants $L_n(f)$ of the target function $f$. Importantly, our algorithm does not require prior knowledge of the curvature of the losses.
\begin{cor}
    Let \(d = 1\) and \(\alpha \in (\frac{1}{2},1]\). Under assumptions of Theorem \ref{theorem:optimal_local_regret}, Algorithm~\ref{alg:Local_Adapt_Algo} achieves a regret with respect to any $f \in \C^\alpha(\X,L), L > 0$, and any pruning $\T \in \Pcal(\T_0)$,
    \vspace{-0.2cm}
    \begin{equation*}
        \textstyle \Reg_T(f) \lesssim  \inf_{\T \in \Pcal(\T_0)} \left\{  \sum_{n \in \L(\T)} 
             \big(L_n(f) |\X_n|^\alpha \big)^{\frac{1}{2\alpha}} \sqrt{|T_n|} \right\} \,.
    \end{equation*}
Moreover, if $\ell_1, \dots, \ell_T$ are exp-concave, one has:
\vspace{-0.2cm}
    \begin{equation*}
         \textstyle \Reg_T(f) \lesssim  \inf_{\T \in \Pcal(\T_0)} \left\{ \textstyle \sum_{n \in \L(\T)} 
               \big(L_n(f)|\X_n|^\alpha \big)^{\frac{2}{2\alpha + 1}} |T_n|^{\frac{1}{2\alpha + 1}}
        \right\} \,.
    \end{equation*}
\label{cor:regret_flat_pruning_exp_concave}
\end{cor}
\vspace{-0.2cm}

The proof of Corollary \ref{cor:regret_flat_pruning_exp_concave} is postponed to Appendix \ref{appendix:proof_corollary:flat_pruning}.
In particular, upper-bounding the infimum over all prunings by the root, our regret becomes $O(L^{2/(2\alpha+1)} T^{1/(2\alpha + 1)})$ and $O(L^{1/(2\alpha)} \sqrt{T})$ for the exp-concave and general case respectively. This achieves the same optimal regret to that obtained in \citet{gaillard2015chaining}, for any sequence of exp-concave losses, without the prior-knowledge of the scale-parameter $\gamma$ that they require, and adapting to any regularity while they consider $L,\alpha = 1$. Our algorithm is also nearly minimax in term of the constants \((L, \alpha)\) as shown by \citet{tsybakov2008introduction}, \citet{hazan2007online} or \citet{bach2024learning}. We provide some experimental illustrations of the results from Corollary \ref{cor:regret_flat_pruning_exp_concave} in Appendix \ref{appendix:xp}.

We note that the fast rate in $T$ obtained under exp-concavity is not optimal in $L$. Thus a compromise is made by our algorithm which competes with more complex oracle trees when $L$ is large to improve and obtain the rate $\sqrt L$ by decreasing the rate in $T$. Such trade-off is classical in parametric online learning as bearing resemblances with the comparison between first and second order algorithms, the first ones being optimal in the dimension, the second ones in $T$. Remarkably, our unique algorithm achieves both regret bounds which opens the door to a minimax theory on rates in $L$ and $T$ and not solely on fast rates in $T$.

\paragraph{Adaptivity to local regularities.}
Theorem \ref{theorem:optimal_local_regret} improves the optimal regret bound established in Theorem \ref{theorem:minimax_CT_param_free} by making it adaptive to the local regularities of the Hölder function $f$. To illustrate this better, applying Hölder's inequality entails (see Appendix \ref{appendix:proof_simplified_bound} for details): for any pruning $\T$
\begin{equation}
\textstyle
    \Reg_T(f) \lesssim   \left\{ 
        \begin{array}{ll}
              (|\X|^\alpha  \bar L(f))^{\frac{2}{2\alpha + 1}} T^{\frac{1}{2\alpha + 1}}  & \text{if $\ell_t$ are exp-concave}\,, \\
             (|\X|^\alpha  \bar L(f))^{\frac{1}{2\alpha}} \sqrt{T}\,, & 
        \end{array}
        \right.
        \label{eq:simple_pruning_bound}
\end{equation}
where $\bar L(f) = \big( \frac{1}{|\X|} \sum_{n \in \L(\T)} |\X_n| L_n(f)^{1/\alpha}\big)^\alpha $ is an average of the local Hölder constants $L_n(f)$ weighted by the size of the sets $\X_n$ over $\T$. This result is in the same spirit as that of \citet{kuzborskij2020locally}, that focus on adapting to tree-based local Lipschitz profiles. However, contrary to us, they need to assume the prior knowledge of bounds $(M^{(k)})_{1\leq k\leq \depth(\T_0)}$ such that $M^{(k)} \geq L_n(f)$ for any $n \in \N(\T_0), \depth(n) = k$. Doing so, for any pruning $\T$, when $\alpha = 1$ and ignoring the dependence on $\X$, for the squared loss (which is exp-concave), they prove a bound of order
\[
    O\Big((\bar M(f) T)^{\frac{1}{2}} + \textstyle \sum_{k} (M^{(k)} |T^{(k)}|)^{1/2} \Big) \qquad \text{where} \qquad \bar M(f) = \textstyle \sum_{k=1}^{\depth(\T)} w^{(k)}  M^{(k)}\,,
\]
with $w^{(k)}$ the proportion of leaves at depth $k$ in the pruning; and $T^{(k)}$ the set of rounds in which $x_t$ belongs to a leaf at level $k$. By grouping our leaves $n$ by their respective depths and applying Hölder's inequality, our results recover theirs with two key improvements (see Appendix \ref{appendix:comparison_with_CB} for details): (1) the prior-knowledge of the $M^{(k)}$ is not required in our case and they are replaced with the true local Hölder constants $L_n(f)$ that are smaller; (2) the rate in $T$ is improved from $\sqrt{T}$ to $T^{1/3}$. Note that, similarly to us, the results of \citet{kuzborskij2020locally} hold for general dimensions and convex losses as well.

\section{Conclusion and perspectives}

In this paper we introduced an online learning approach based on chaining trees and proved that this method achieves minimax regret for the $\alpha$-Hölder nonparametric regression problem, $\alpha \in (0,1]$. We designed a general and computationally tractable algorithm that leverages a core structure based on chaining-trees to perform an optimal local approximation of $\alpha$-Hölder functions, where $\alpha \leq 1$. In addition, we showed that our algorithm adapts to the curvature of the loss functions revealed by the environment, while remaining optimal in a minimax sense.  
A limitation of our approach is that chaining trees are minimax-competitive only against $\alpha$-Hölder continuous functions when the smoothness parameter $\alpha \in (0,1]$. However, combinations of trees, such as the forests studied in \citet{arlot2014analysis} and \citet{mourtada2020minimax}, achieve minimax rates for $\alpha \in (1,2]$. Since their framework is based on batch i.i.d. data, an open question remains as to whether combinations of chaining trees can also be minimax-optimal in an adversarial setting against functions with higher regularity.  

As future work, our approach could be extended to incorporate alternative structures beyond chaining trees, such as kernels or shallow networks. In particular, employing other function approximation methods could address a nonparametric regression problem with respect to richer classes of functions. 

\paragraph{Link with boosting in adversarial online regression.} 

Boosting is a well-established strategy in statistical learning \citep{friedman2001greedy, zhang2005boosting}, where a set of weak learners is iteratively combined to construct a strong predictor with improved accuracy. Conceptually, this process refines predictions at each step by correcting errors from previous iterations. Our approach shares similarities with boosting-based methods in that it iteratively and adaptively refines function approximations over time. For instance, the structure of chaining trees that we studied can be seen as an implicit hierarchical refinement process, akin to boosting’s combination of weak learners.  
While boosting has been extensively studied in batch settings, recent research \citep{beygelzimer2015online, hazan2021boosting} has encouraged the adaptation and study of boosting procedures in the context of adversarial nonparametric regression.  

\begin{wrapfigure}{r}{0.35\textwidth}
\begin{center} \vspace{-1cm}
\begin{tikzpicture}[every node/.style={scale=0.8},scale=0.8]
  \tikzstyle{estun}=[->,>=latex,blue,dotted,line width=1pt]
   \node[left] (a) at (0,0) {\(\beta_{1,t}h_{1,t}\)};
   \node[left] (b) at (0,-1.5) {\(\beta_{n-1,t}h_{n-1,t}\)};
   \node[left] (p) at (0,-0.75) {\(\vdots\)};
   \node[left] (d) at (0,-2.5) {\(\beta_{n,{\color{red} t}}h_{n,{\color{red} t}}\)};
   \node[left] (v) at (0,-4.25) {\(\vdots\)};
   \node[left] (e) at (0,-3.5) {\(\beta_{n+1,t}h_{n+1,t}\)};
   \node[left] (f) at (0,-5) {\(\beta_{N,t}h_{N,t}\)};
   \node (m1) at (1.43,-2.55) {};
   \node (m) at (1.42,-2.45) {};
   \node (d') at (3.,-2.5) {\(\beta_{n,{\color{red} t+1}}h_{n,{\color{red} t+1}}\)};
   \draw[estun] (a) to[in=80, out=0,] (1.375-0.2,-2.1);
   \draw[estun] (b) to[in=110, out=0] (1.075-0.2,-2.1);
   \draw[estun] (e) to[in=-110, out=0] (1.075-0.2,-2.9);
   \draw[estun] (f) to[in=-80, out=0] (1.375-0.2,-2.9);
   \draw[blue, line width=1pt] (1.,-2.5) circle (0.4);
   \draw[->,>=latex,line width=1pt] (1.4,-2.5)--(d') node[midway,above]{};
   \draw[line width=1pt] (d)--(0.6,-2.5) node[midway,above]{};
   \begin{scriptsize}
   {\color{blue}
   \draw(1.125-0.2,-1.5)node[below,rotate=35]{$\dots$};
   \draw(1.125-0.2,-3.5)node[above, rotate=-35]{$\dots$};
   \draw(1.,-2.5) node {$g_{n,t}$};}
   \end{scriptsize}
\end{tikzpicture}   
\end{center}
\caption{Boosting at time \(t\).}
\label{fig:diag_boost}
\end{wrapfigure} 

A natural question is whether exposing our algorithms at a meta-state could provide a foundation for analyzing more general weak learners in the context of adversarial online regression. Specifically, instead of relying on a pre-defined hierarchical structure such as chaining trees, one could explore dynamically learning general weak function approximators (e.g., shallow trees, shallow networks) and adaptively aggregating them over time. This perspective is motivated by a more general form of our Algorithm \ref{alg:training_CT}, which we expose here.  

Let $\W$ be a set of real-valued functions $\X \to \R$, and for some $N \ge 1$, define the function space:  
\begin{equation}
    \label{eq:span}
    \operatorname{span}_N(\W) = \big\{\textstyle \sum_{n=1}^N \beta_n h_n, \; h_n \in \W, \beta_n \in \R\big\},
\end{equation}
which forms a linear space of functions based on $N$ elements from $\W$. The goal is to find a sequence of functions $\hat f_t \in \operatorname{span}_N(\W)$, for $t \geq 1$, such that it minimizes the regret $\Reg_T(\F)$ as defined in \eqref{eq:reg}, with $\F = \operatorname{span}_N(\W)$.  

To illustrate this general perspective, one could present our Algorithm \ref{alg:training_CT} as an abstract formulation of a boosting-like procedure for function approximation based on a gradient update. A schematic diagram is provided in Figure \ref{fig:diag_boost}. Specifically, at each step $t \geq 1$, a meta-version of Algorithm \ref{alg:training_CT} would perform Equation \eqref{eq:grad-step_CT} to find a pair $(\beta_{n,t+1},h_{n,t+1}) \in \R \times \W$ approximating a minimum of the following objective function
\begin{equation}
\label{eq:boosting}
 (\beta_n,h_n) \mapsto \ell_t(\hat f_{-n,t}(x_t) + \beta_n h_n(x_t)) \quad \text{where} \quad \hat f_{-n,t}(x_t) = \hat f_t(x_t) - \beta_{n,t}h_{n,t}(x_t)
\end{equation}
using the gradient  $  \left[\nabla_{(\beta_n,h_n)} \ell_t\big(\hat f_{-n,t}(x_t) + \beta_n h_n(x_t)\big)\right]_{(\beta_n,h_n) = (\beta_{n,t}, h_{n,t})}$.

In Section \ref{section:minimax_CT}, we analyzed the special case where $\W$ is specified as $\{h_n : x \mapsto \theta_n \ind{x\in \X_n}, \theta_n \in \R, n \in \N(\T)\}$, with fixed $\beta_n =1$ and $N = |\N(\T)|$, using a parameter-free gradient minimization step. A compelling direction for future research is to analyze whether the meta-algorithm defined by Equation \eqref{eq:boosting} can achieve minimax rates under assumptions on weak learners belonging to a general $\W$. By framing the problem in this way, we believe that it could be analyzed more broadly within an online and adversarial boosting framework - see, for instance, \citet{beygelzimer2015online}.




  

\acks{We acknowledge the financial and material support of the Sorbonne Center for Artificial Intelligence (SCAI) currently funding the Ph.D scholarship of Paul Liautaud. Pierre Gaillard and Paul Liautaud also thank Institut Pauli CNRS, Wien, for their hospitality during their visits.}

\bibliography{biblio}

\newpage
\appendix
\begin{center}
    \Large APPENDIX
\end{center}


\section{Proof of Theorem \ref{theorem:minimax_CT_param_free}}

\label{appendix:proof_theorem:minimax_CT_param_free}

Let $f^* \in \argmin_{f \in \C^\alpha(\X,L)} \sum_{t=1}^T \ell_t(f(x_t))$. We define the function  
\begin{equation}
    \label{eq:fW}
   \hat f^* = \sum_{n \in \L(\T)} f^*(x_n) \ind{\X_n}   \,, 
\end{equation}
where $T_n = \{1\leq t\leq T: x_t \in \X_n\}$, $x_n$ the center of hyper region $\X_n$ (i.e. for any $x \in \X_n, \|x-x_n\|\leq 2^{-1}|\X_n|$). The proof starts with the following regret decomposition 
\begin{equation}
    \label{eq:decomposition}
    \Reg_T(\C^\alpha(\X,L)) = \underbrace{\sum_{t=1}^T \ell_t(\hat f_t(x_t)) - \ell_t(\hat f^*(x_t))}_{R_1} +  \underbrace{\sum_{t=1}^T \ell_t(\hat f^*(x_t)) - \ell_t(f^*(x_t))}_{R_2} \,.
\end{equation}
We will refer to $R_1$ as the estimation error, which consists of the error incurred by sequentially learning the best Chaining-Tree $\hat f^*$. $R_2$ will refer to the approximation error, which involves approximating Hölder functions in $\C^\alpha(\X,L)$ by piecewise-constant functions with $|\L(\T)|$ pieces.

\paragraph{Step 1: Upper-bounding the approximation error $R_2$.}  Note that by definition of the Chaining-Tree~$\T$ (see Definition~\ref{def:chaining_tree}), $\{\X_n, n\in \L(\T)\}$ forms a partition of $\X = \X_{\root(\T)}$ and for any leaf $n \in \L(\T)$
\begin{equation}
    \label{eq:leaf_diameter}
    |\X_n| = \frac{|\X_{\root(\T)}|}{2^{\depth(n)-1}} = \frac{|\X|}{2^{\depth(\T)-1}} \,.
\end{equation}
Then,
\begin{align}
    R_2 &= \sum_{t=1}^T \ell_t(\hat f^*(x_t)) -  \ell_t(f^*(x_t)) &\nonumber\\
     &\leq \sum_{t=1}^T G |\hat f^*(x_t) - f^*(x_t)|  & \leftarrow \text{$\ell_t$ is $G$-Lipschitz} \nonumber \\
     & = G  \sum_{t=1}^T \Big| \sum_{n \in \L(\T)} f^*(x_n)\ind{x_t \in \X_n} - f^*(x_t)\Big| & \leftarrow \text{by~\eqref{eq:fW}} \nonumber \\ 
    & = G \sum_{n \in \L(\T)} \sum_{t \in T_n} |f^*(x_n) - f^*(x_t)| & \leftarrow \text{$\{\X_n, n \in \L(\T)\}$ partitions $\X$} \nonumber \\
    & \leq G \sum_{n \in \L(\T)} \sum_{t \in T_n} L \|x_n - x_t\|_{\infty}^\alpha & \leftarrow f^* \in \C^\alpha(\X,L) \nonumber \\
    & \leq G \sum_{n \in \L(\T)} L 2^{-\alpha}|\X_n|^\alpha |T_n|  & \leftarrow x_n \text{ center of } \X_n \nonumber \\
    & \leq  G L 2^{-\alpha \depth(\T)}|\X|^\alpha T  \,, &  \label{eq:approximation_error}
\end{align}
where the last inequality is by~\eqref{eq:leaf_diameter} and because the leaves form a partition of $\X$, which implies $\sum_{n\in \L(\T)} |T_n| = T$.

\paragraph{Step 2: Upper-bounding the estimation error $R_1$.} We now turn to the bound of the estimation error, that is the regret with respect to best Chaining-Tree $\hat f^*$.

\emph{Step 2.1: Parametrization of $\hat f^*$ in terms of $\theta_n$}. Note that the parametrization of $\hat f^*$ in terms of $\theta_n$ is non-unique. We design below a parametrization such that for any $x \in \X$
\begin{equation}
    \label{eq:fW2}
    \hat f^*(x) =  \sum_{n \in \N(\T)} \theta_n \ind{x \in \X_n} \,,
\end{equation}
and which will allow us to leverage the chaining structure of our Chaining-Tree.
We define, 
\begin{equation}
    \label{eq:thetan}
    \theta_{\root(\T)} = f^*(x_{\root(\T)}) \quad \text{and} \quad \theta_{n} = f^*(x_n) -  f^*(x_{\p(n)}), \quad \text{for } n\neq \root(\T) \,,
\end{equation}
where $T_n = \{1\leq t\leq T: x_t \in \X_n\}$ and $x_n$ stands for the center of subregion $\X_n$ for any $n\in \N(\T)$. \\
Let us show that the above construction \eqref{eq:thetan} indeed satisfies~\eqref{eq:fW2}. To do so, we fix $x \in \X$ and proceed by induction on $m=1,\dots,\depth(\T)$, by proving that
\begin{equation}
    \label{eq:Hn}
    \tag{$\mathcal{H}_m$}
    \sum_{n \in \N(\T)} \theta_n \ind{x \in \X_n} \ind{\depth(n) \leq m} = \sum_{n \in \N(\T)} f^*(x_n) \ind{x \in \X_n} \ind{\depth(n) = m} \,.
\end{equation}
First, note that ($\mathcal{H}_1$) is true by definition of $\theta_{\root(\T)}$. Then, let $m \geq 1$, and assume that ($\mathcal{H}_m$) is satisfied, we have
\begin{align*}
     \sum_{n \in \N(\T)} & \theta_n \ind{x \in \X_n} \ind{\depth(n) \leq m+1}& \\
         &=  \sum_{n \in \N(\T)} \theta_n \ind{x \in \X_n} \ind{\depth(n) \leq m} +  \sum_{n \in \N(\T)} \theta_n \ind{x \in \X_n} \ind{\depth(n) = m+1} &  \\
        & = \sum_{n \in \N(\T)} f^*(x_n) \ind{x \in \X_n} \ind{\depth(n) = m} + \sum_{n \in \N(\T)} \theta_n \ind{x \in \X_n} \ind{\depth(n) = m+1} & \leftarrow \text{by ($\mathcal{H}_m$)}  \\
        & = \sum_{n \in \N(\T)} f^*(x_n)\ind{x \in \X_n} \ind{\depth(n) = m} \\
        & \hspace*{2cm} + \sum_{n \in \N(\T)} ( f^*(x_n) -  f^*(x_{\p(n)}) ) \ind{x \in \X_n} \ind{\depth(n) = m+1}
          & \leftarrow \text{by~\eqref{eq:thetan}} \\
        & = \sum_{n \in \N(\T)} f^*(x_n) \ind{x \in \X_n} \ind{\depth(n) = m+1} \,,
\end{align*}
which concludes the induction. In particular, for $m = \depth(\T)$, ($\mathcal{H}_m$) yields
\[
    \sum_{n \in \N(\T)} \theta_n \ind{x \in \X_n} = \sum_{n \in \L(\T)} f^*(x_n) \ind{x \in \X_n}  = \hat f^*(x)\,,
\]
where the last equality is by definition of $\hat f^*$ in~\eqref{eq:fW}. 

\emph{Step 2.2: Upper-bounding $|\theta_n|$}. The key advantage of the parametrization $\theta_n$ in~\eqref{eq:thetan} is that it leverages the chaining structure of our tree. Each node aims to correct the error made by its parent, and as we show below, this error decreases significantly with the depth $\depth(n)$ of the node $n$. Let $n \in \N(\T)\setminus \{\root(\T)\}$, 
\begin{equation}
    \label{eq:bound_thetan}
    |\theta_n| 
          = |f^*(x_n) -  f^*(x_{\p(n)})| 
          \leq   L \|x_n - x_{\p(n)}\|^\alpha_\infty = L 2^{-\alpha}|\X_n|^\alpha = L|\X|^\alpha 2^{-\alpha \depth(n)}
\end{equation}
where the last equalities are because $\X_n \subset \X_{\p(n)}$ and $|\X_n| = |\X| 2^{-(\depth(n)-1)}$, from Definition~\ref{def:chaining_tree}. Furthermore, by definition of $\C^\alpha(\X,L)$, $|f^*(x)| \leq B$ for any $x \in \X$, hence \[|\theta_{\root(\T)}| = |f^*(x_{\root(\T)})| \leq B\,.\] 

\emph{Step 2.3: Proof of the regret upper bound.} We are now ready to upper bound the estimation error in~\eqref{eq:decomposition}. We have
\begin{align}
    R_1 &= \sum_{t=1}^T \ell_t(\hat f_t(x_t)) - \ell_t(\hat f^*(x_t)) \nonumber\\
        & = \sum_{t=1}^T \textstyle \ell_t\big(\sum_{n \in \N(\T)} \theta_{n,t} \ind{x_t \in \X_n}\big) - \ell_t\big(\sum_{n \in \N(\T)} \theta_n \ind{x_t \in \X_n}\big)  \nonumber \\
        & \leq \sum_{t=1}^T \sum_{n \in \N(\T)} g_{n,t}(\theta_{n,t} - \theta_n)
\end{align}
by convexity of $\ell_t$, where $g_{n,t}$ is the partial subgradient in $\theta_{n,t}$ as defined in Equation~\eqref{eq:gradient_CT}. Now, from Assumption~\ref{assumption:parameter_free} on the \texttt{grad-step} procedure to optimize $\theta_{n,t}$ and with $\theta_{n,1}=0, g_{n,t} \leq G\ind{x_t \in \X_n}$, we further have, with $T_n = \{1 \leq t \leq T : g_{n,t} \neq 0\}$,
\begin{align}
    R_1 & \leq  G \sum_{n \in \N(\T)} |\theta_n| (C_1 \sqrt{|T_n| } +  C_2) \nonumber \\
        & = G \sum_{m=1}^{\depth(\T)} \sum_{n:\depth(n)=m}|\theta_n| (C_1 \sqrt{|T_n| } +  C_2)  & \nonumber \\
        & \leq BG(C_1\sqrt{T} + C_2) +  L G |\X|^\alpha  \sum_{m=2}^{\depth(\T)} \sum_{n:\depth(n)=m}  (C_1 \sqrt{|T_n| } +  C_2) 2^{-\alpha m}    \hspace*{2cm} \leftarrow \text{by~\eqref{eq:bound_thetan}} \label{eq:decomposition_reg_estimation}
\end{align}
Now, because in a \(d\)-regular decision tree, the number of nodes with depth $m$ equals $|\{n: \depth(n) = m\}| = 2^{d(m-1)}$ (recall that the depth of the root is 1), and because $\{\X_n: \depth(n) = m\}$ forms a partition of $\X$, we have $\sum_{n:\depth(n)=m} T_n = T$ and by Cauchy-Schwarz inequality
\[
    \sum_{n:\depth(n)=m} \sqrt{T_n} \leq \sqrt{ 2^{d(m-1)} {\textstyle \sum_{n:\depth(n)=m}} T_n}  = \sqrt{2^{d(m-1)} T} \,,
\]
which substituted into the previous upper bound entails
\begin{align}
    R_1 
        & \leq BG(C_1\sqrt{T} + C_2) + L G |\X|^\alpha  \sum_{m=2}^{\depth(\T)} \Big( C_1  2^{\frac{d(m-1)}{2} - \alpha m} \sqrt{T} +  C_2  2^{d(m-1) - \alpha m}   \Big) \nonumber \\
        & = BG(C_1\sqrt{T} + C_2) + LG|\X|^\alpha \bigg(2^{-\frac{d}{2}}C_1 \sqrt{T} \sum_{m=2}^{\depth(\T)} 2^{m(\frac{d}{2} - \alpha )} + 2^{-d} C_2 \sum_{m=2}^{\depth(\T)} 2^{m (d-\alpha)} \bigg) \,.
        \label{eq:estimation_error}
\end{align}

\paragraph{Step 3: Conclusion and optimization of $\depth(\T)$.} To conclude the proof, we consider three cases according to the sign of $d-2\alpha$:

\emph{$\bullet$ Case 1: if $d < 2\alpha$}. Then
    \[
       2^{-\frac{d}{2}}\sum_{m=2}^{\depth(\T)} 2^{m(\frac{d}{2} - \alpha )} \leq \frac{1}{1-2^{\frac{d}{2}-\alpha}}  \quad \text{ and } \quad   2^{-d}\sum_{m=2}^{\depth(\T)} 2^{m (d-\alpha)}   \leq 2^{-d}\sum_{m=0}^{\depth(\T)} 2^{m\alpha} \leq 2^{-d} \frac{2^{\alpha (\depth(\T)+1)}}{2^\alpha - 1} \stackrel{(2\alpha \geq 1)}{\leq} 2^{\alpha \depth(\T) + 2}  \,,
    \] 
    and~\eqref{eq:estimation_error} yields 
    \[
         R_1 \leq BG(C_1\sqrt{T} + C_2) + LG|\X|^\alpha \bigg( \frac{C_1 \sqrt{T} }{1-2^{\frac{d}{2}-\alpha}} + C_2 2^{\alpha \depth(\T) + 2} \bigg) \,;
    \]
    Therefore, combining with~\eqref{eq:decomposition} and~\eqref{eq:approximation_error}, the regret is upper-bounded as
    \[
        \Reg_T(\C^\alpha(\X,L)) \leq BG(C_1\sqrt{T} + C_2) + LG |\X|^\alpha \Big(  \frac{C_1 \sqrt{T}}{1-2^{\frac{d}{2}-\alpha}} + C_2 2^{\alpha \depth(\T) + 2} + T 2^{-\alpha \depth(\T)}\Big) \,.
    \]
    The choice $\depth(\T) = \frac{1}{d} \log_2 T$ entails
    \begin{equation}
        \label{eq:case1}
        \Reg_T(\C^\alpha(\X,L)) \leq BG(C_1\sqrt{T} + C_2) + LG |\X|^\alpha \Big(  \frac{C_1}{1-2^{\frac{d}{2}-\alpha}} + 4 C_2 + 1 \Big) \sqrt{T} \,.
    \end{equation}

\emph{$\bullet$ Case 2: if $d = 2\alpha$}. Then
    \[
       2^{-\frac{d}{2}}\sum_{m=2}^{\depth(\T)} 2^{m(\frac{d}{2} - \alpha )} \leq \depth(\T) \quad \text{ and } \quad   2^{-d}\sum_{m=2}^{\depth(\T)} 2^{m (d-\alpha)}  = 2^{-d} \sum_{m=2}^{\depth(\T)} 2^{m\alpha} \leq 2^{\alpha \depth(\T) + 2}  \,,
    \] 
    and~\eqref{eq:estimation_error} yields
    \[
        R_1 \leq BG(C_1\sqrt{T} + C_2) + LG|\X|^\alpha \Big( C_1 \sqrt{T} \depth(\T) + C_2 2^{\alpha \depth(\T) + 2} \Big) \,;
    \]
    Therefore, combining with~\eqref{eq:decomposition} and~\eqref{eq:approximation_error}, the regret is upper-bounded as
    \[
        \Reg_T(\C^\alpha(\X,L)) \leq BG(C_1\sqrt{T} + C_2) + 2^\alpha LG |\X|^\alpha \Big(  C_1 \sqrt{T} \depth(\T) + C_2 2^{\alpha \depth(\T) + 2} + T 2^{-\alpha \depth(\T)}\Big) \,.
    \]
    The choice $\depth(\T) = \frac{1}{d} \log_2 T$ entails
    \begin{equation}
        \label{eq:case2}
        \Reg_T(\C^\alpha(\X,L)) \leq BG(C_1\sqrt{T} + C_2) + 2^\alpha LG |\X|^\alpha \Big(  \frac{C_1}{d} \log_2 T + 4 C_2 + 1 \Big) \sqrt{T} \,.
    \end{equation}

\emph{$\bullet$ Case 3: if $d > 2\alpha$}. Then
    \[
       2^{-\frac{d}{2}}\sum_{m=2}^{\depth(\T)} 2^{m(\frac{d}{2} - \alpha )} \leq \frac{2^{(\frac{d}{2}-\alpha)\depth(\T)}}{2^{\frac{d}{2}-\alpha} - 1} \quad \text{ and } \quad   2^{-d}\sum_{m=2}^{\depth(\T)} 2^{m (d-\alpha)}  \leq \frac{2^{(d-\alpha) \depth(\T) } }{2^{d-\alpha} - 1} \leq 2^{(d-\alpha) \depth(\T) + 2 } \,,
    \] 
    where the last inequality is because $(2^{d-\alpha} -1)^{-1} \leq (2^{d/2} -1)^{-1} \leq (\sqrt{2}-1)^{-1} \leq 4$. And~\eqref{eq:estimation_error} yields
    \[
         R_1 \leq BG(C_1\sqrt{T} + C_2) + LG|\X|^\alpha \bigg( C_1 \sqrt{T} \frac{2^{ (\frac{d}{2}-\alpha) \depth(\T)}}{2^{\frac{d}{2}-\alpha}-1} +  C_2 2^{(d-\alpha)\depth(\T) + 2}  \bigg) \,.
    \]
    Therefore, combining with~\eqref{eq:decomposition} and~\eqref{eq:approximation_error}, the regret is upper-bounded as
    \[
        \Reg_T(\C^\alpha(\X,L)) \leq BG(C_1\sqrt{T} + C_2) + LG |\X|^\alpha \bigg( C_1 \sqrt{T} \frac{2^{ (\frac{d}{2}-\alpha) \depth(\T)}}{2^{\frac{d}{2}-\alpha}-1} +  C_2 2^{(d-\alpha)\depth(\T) + 2} + T 2^{-\alpha \depth(\T)}\bigg) \,.
    \]
    The choice $\depth(\T) = \frac{1}{d} \log_2 T$ entails
    \begin{equation}
        \label{eq:case3}
        \Reg_T(\C^\alpha(\X,L)) \leq BG(C_1\sqrt{T} + C_2) + LG |\X|^\alpha \bigg(  \frac{ C_1}{2^{\frac{d}{2}-\alpha}-1}  + 4 C_2   + 1 \bigg) T^{1- \frac{\alpha}{d}} \,.
    \end{equation}

\emph{Conclusion.} Combining the three cases~\eqref{eq:case1},~\eqref{eq:case2}, and~\eqref{eq:case3} concludes the proof of the regret bound, which we summarize below
    \[
      \Reg_T(\C^\alpha(\X,L)) \leq BG(C_1\sqrt{T} + C_2) + GL|\X|^\alpha 
        \begin{cases}             
            \big(  \Phi(\frac{d}{2}-\alpha) C_1 + 4 C_2 + 1 \big) \sqrt{T} & \text{if $d< 2\alpha$} \\ 
            \big(\frac{C_1}{d} \log_2 T + 4 C_2 + 1 \big) \sqrt{T} & \text{if $d=2\alpha$}  \\
            \big( \Phi(\frac{d}{2}-\alpha) C_1 + 4 C_2 + 1 \big) T^{1-\frac{\alpha}{d}} & \text{if $d > 2\alpha$} \,, 
    \end{cases}
    \]
    where $\Phi(u) = |2^{u} - 1|^{-1}$.

\newpage 
\section{Proof of Theorem \ref{theorem:optimal_local_regret}}
\label{appendix:proof_theorem:optimal_local_regret}

We state here the full version of Theorem \ref{theorem:optimal_local_regret} that we prove right after.
\begin{theorem}
\label{theorem:full_version_Locally_Adaptive_Reg}
    Let \(T, d \geq 1\) and $(\T_0, \bar \X, \bar \W)$ be a core regular tree with CT $\{\T_{n,k}, (n,k) \in \N(\T_0)\times [K]\}$ satisfying the same assumptions as in Theorem \ref{theorem:minimax_CT_param_free} and root nodes initialized as $\theta_{\root(\T_{n,k}),1} = \gamma_k \in \Gamma$, for all $(n, k) \in \N(\T_0) \times [K]$.
    Then, Algorithm~\ref{alg:Local_Adapt_Algo} with a \emph{\texttt{weight}} subroutine as in Assumption~\ref{assumption:second_order_algo}, achieves the regret upper-bound with respect to any $f \in \C^\alpha(\X,L), L > 0$,
    \begin{multline*}
        \Reg_T(f) \leq \inf_{\T \in \Pcal(\T_0)} \Bigg\{ \beta_1 \sqrt{T|\L(\T)|} + \beta_2|\L(\T)|
         \\ + G|\X|^\alpha \sum_{n \in \L(\T)} L_n(f) 2^{-\alpha(\depth(n)-1)} \begin{cases}
        \psi_1 \sqrt{|T_n|} & \text{if \(d < 2\alpha\)} \\ \psi_2 \log_2|T_n|\sqrt{|T_n|} & \text{if \(d = 2\alpha\)} \\ \psi_1 |T_n|^{1-\frac{\alpha}{d}} & \text{if \(d > 2\alpha\)},
    \end{cases} \Bigg\},
    \end{multline*}
with $\beta_1 = 2C_3G\sqrt{\log\big(2BT|\N(\T_0)|)}$ and $\beta_2 = G(2^{-1}C_1 + C_2 2^{-1} T^{-\frac{1}{2}} + C_4)$, local Lipschitz constants \(L_n(f) \leq L\) as in \eqref{eq:deflochold}, \(\psi_1 = \Phi(d/2 - \alpha)C_1 + 4C_2 +1, \psi_2 = C_1/d + 4C_2 +1\), and \(\Phi, C_1, C_2\) as in Theorem \ref{theorem:minimax_CT_param_free}.

Moreover, if \(\ell_1, \dots, \ell_T\) are \(\eta\)-exp-concave with some \(\eta > 0\), one has:
\begin{equation*}
        \Reg_T(f) \leq \inf_{\T \in \Pcal(\T_0)} \Bigg\{ \beta_3|\L(\T)| + G|\X|^\alpha \sum_{n \in \L(\T)} L_n(f) 2^{-\alpha(\depth(n)-1)} 
    \begin{cases}
        \psi_1 \sqrt{|T_n|} & \text{if \(d < 2\alpha\)} \\ 
        \psi_2 \log_2|T_n|\sqrt{|T_n|} & \text{if \(d = 2\alpha\)} \\ \psi_1 |T_n|^{1-\frac{\alpha}{d}} & \text{if \(d > 2\alpha\)},
    \end{cases}
    \Bigg\}
\end{equation*}
with $\beta_3 = \frac{C_3^2 \log\big(2 BT|\N(\T_0)|\big)}{2\mu} + C_4G + 2^{-1}G (C_1 + C_2T^{- \frac{1}{2}})$ and \(0 < \mu \leq \min\{1/G,\eta\}\).
\end{theorem}

\begin{proof}[{\bfseries of Theorem \ref{theorem:full_version_Locally_Adaptive_Reg}}]
Let $L>0, \alpha \in (0,1], f^* \in \C^\alpha(\X,L)$ and $\varepsilon >0$ be the precision of the grid $\Gamma$, $K = \lfloor 2B/\varepsilon \rfloor$ the number of experts in each node in $\N(\T_0)$. Let $\T \in \Pcal(\T_0)$ be some pruned tree from $(\T_0, \bar \X, \bar \W)$ with prediction functions \(\bar \W = \{(\hat f_{n,k})_{k \in [K]}, n \in \N(\T_0)\}\) on subsets \(\bar \X = \{\X_n, n \in \N(\T_0)\}\). We call \(\hat f_{\T}\) the associated prediction function of pruning \(\T\) (see Definition \ref{def:pruning}) such that at any time $t \geq 1$, \[\hat f_{\T,t}(x) = \sum_{n \in \L(\T)} \hat f_{n,k_n,t}(x)\,, \qquad x \in \X,\] 
with $k_n = \argmin_{k \in [K]} |(-B+(k-1)\varepsilon) - f^*(x_{n})|$ the best approximating constant of $f^*(x_{n})$ where $x_{n} \in \X_n$ is the center of the sub-region \(\X_n\), i.e. for any $x \in \X_n, \|x-x_n\|_\infty \leq 2^{-1}|\X_n|$.
We have a decomposition of regret as:
\begin{equation}
\Reg_T(f) = \underbrace{\sumT \ell_t(\hat f_t(x_t)) - \ell_t(\hat f_{\T,t}(x_t))}_{=: R_1} + \underbrace{\sumT \ell_t(\hat f_{\T,t}(x_t)) - \ell_t(f^*(x_t))}_{=: R_2},
\label{eq:decomposition_reg}
\end{equation}
\(R_1\) is the regret related to the estimation error of the core expert tree \(\T_0\) compared to some pruning \(\T\) from it. 
On the other hand, \(R_2\) is related to the error of the pruning tree \(\T\) against some function \(f^*\). 

\paragraph{Step 1: Upper-bounding \(R_2\) as local chaining tree regrets.} Recall that according to Definition \ref{def:pruning}, pruning subsets \(\{\X_n, n \in \L(\T)\}\) form a partition of \(\X = \X_{\mathrm{root}(\T_0)}\). Hence, for any \(x_t \in \X\), prediction from pruning \(\T\) at time $t$ is \(\hat f_{\T,t}(x_t) = \hat f_{n,k_n,t}(x_t)\) with $n \in \L(\T)$ the unique leaf such that \(x_t \in \X_{n}\) at time $t$. Then, \(R_2\) can be written as follows:
\begin{align}
        R_2 &= \sumT \sum_{n \in \L(\T)} (\ell_t(\hat f_{\T,t}(x_t)) - \ell_t(f^*(x_t)))\ind{x_t \in \X_n} \nonumber \\
        & = \sum_{n \in \L(\T)} \sum_{t\in T_n}  \ell_t(\hat f_{n,k_n,t}(x_t)) - \ell_t(f^*(x_t)) \nonumber \\
        & \leq \sum_{n \in \L(\T)} \sum_{t\in T_n}  \ell_t(\tilde  f_{n,k_n,t}(x_t)) - \ell_t(f^*(x_t)) 
        , \label{eq:node_approximation}
\end{align}
    where we set \(T_n = \{1 \le t \le T : x_t \in \X_n\}, n \in \L(\T)\) and \eqref{eq:node_approximation} is because $\hat f_{n,k_n,t} = [\tilde f_{n,k_n,t}]_B \leq \tilde f_{n,k_n,t}$ and $\ell_t$ is convex and has minimum in $[-B,B]$.

The decomposition in \eqref{eq:node_approximation} represents a sum of \emph{local} error approximations of the function \(f^*\) over the partition \(\{\X_n, n \in \L(\T)\}\), using predictors \(\tilde {f}_{n,k_n}\) located at the leaves of the pruned tree \(\T\). Recall that for every \(n \in \N(\T_0)\), \(\tilde {f}_{n,k_n}\) is a prediction function associated with a CT $\T_{n,k_n}$, where the root node starts from \(\theta_{\root(\T_{n,k_n}),1} = -B + (k_n - 1)\varepsilon \in \Gamma\) on \(\X_n\). 
In proof of Theorem \ref{theorem:minimax_CT_param_free} (Appendix \ref{appendix:proof_theorem:minimax_CT_param_free}) we study a regret bound \eqref{eq:decomposition} decomposed into two terms: estimation and approximation. In particular, we showed that any CT adapts to any regularity \((L,\alpha) \in \R_+ \times (0,1]\) of \(f^*\). Thus, the approximation error of CT \(\tilde {f}_{n,k_n}\) with respect to \(f^*\) remains similar to that in \eqref{eq:approximation_error}, but now with regard to an Hölder function with a constant \(L_n(f^*) \geq 0\) over \(\X_n\). Specifically, from \eqref{eq:node_approximation}, we get: \begin{multline} R_2 \leq  \sum_{n \in \L(\T)} \left[ \underbrace{G \sum_{m = 1}^{\depth(\T_{n,k_n})} \sum_{n' : \depth(n')=m} |\theta_{n'}-\theta_{n',1}|(C_1\sqrt{|T_{n'}|} + C_2)}_{\text{estimation error as in \eqref{eq:decomposition_reg_estimation}}} \right. \\ \left. + \underbrace{G L_n(f^*) |\X_n|^\alpha |T_n| 2^{-\alpha (\depth(\T_{n,k_n}))}}_{\text{approximation error \eqref{eq:approximation_error} over $\X_n$}}\right] , \label{eq:pre_opt} \end{multline} 
with \(C_1,C_2\) as in Assumption \ref{assumption:parameter_free} and where we set in \eqref{eq:thetan}, \[\theta_{\root(\T_{n,k_n})} = f^* (x_{\root(\T_{n,k_n})}) \quad  \text{and} \quad \theta_{n'} = f^*(x_{n'}) - f^*(x_{\p(n')}), \quad n' \in \N(\T_{n,k_n})\setminus \{\root(\T_{n,k_n})\}.\]
In particular, we have for $n'=\root(\T_{n,k_n})$, \begin{equation}
\label{eq:theta_bounding}
|\theta_{n'}-\theta_{n',1}| = \left| f^*(x_{\root(\T_{n,k_n})}) - (-B + (k_n-1))\varepsilon \right| \leq \frac{\varepsilon}{2},
\end{equation}
by definition of $k_n$ and since $\Gamma = \{-B+(k-1)\varepsilon\}_{k \in [K]}$ is an $\varepsilon$-discretization of the $y$-axis. Moreover, if $n' \in \N(\T_{n,k_n}), \depth(n') \geq 2$, one has $\theta_{n',1} = 0$ and \begin{equation}
\label{eq:theta_bounding_2}
     |\theta_{n'} - \theta_{n',1}| = |\theta_{n'} | \leq L_n(f^*)|\X_n|^\alpha 2^{-\alpha \depth(n')},
\end{equation} according to \eqref{eq:bound_thetan} with $f^* \in \C^\alpha(\X_n,L_n)$.

Then, following the same optimization steps as for Theorem \ref{theorem:minimax_CT_param_free}, in each $\depth(\T_{n,k_n}), n \in \L(\T)$ of \eqref{eq:pre_opt}, we  get: \begin{multline*}
    R_2 \leq  G \sum_{n \in \L(\T)} \frac{\varepsilon}{2} (C_1 \sqrt{|T_{n}|} + C_2)  \\ 
    + G|\X|^\alpha \sum_{n \in \L(\T)} L_n(f^*)2^{-\alpha(\depth(n) - 1)} \begin{cases} \psi_1 \sqrt{|T_n|} & \text{if \(d < 2\alpha\)} \\ \psi_2 \log_2|T_n|\sqrt{|T_n|} & \text{if \(d = 2\alpha\)} \\ \psi_1 |T_n|^{1-\frac{\alpha}{d}} & \text{if \(d > 2\alpha\)}\end{cases}
\end{multline*}
with \(\psi_1 = \Phi(d/2 - \alpha)C_1 + 4C_2 + 1, \psi_2 = C_1/d + 4C_2 +1\), and \(\Phi\) defined in Theorem \ref{theorem:minimax_CT_param_free}.

Cauchy-Schwarz inequality gives \[\sum_{n \in \L(\T)} (C_1\sqrt{|T_n|} + C_2) \leq  C_1\sqrt{|\L(\T)|T} + C_2|\L(\T)|\]

Finally,
 \begin{multline}
     R_2 \leq \frac{\varepsilon}{2} G  \left(C_1\sqrt{|\L(\T)|T} + C_2 |\L(\T)|\right) \\
     \qquad + G|\X|^\alpha \sum_{n \in \L(\T)} L_n(f^*)2^{-\alpha(\depth(n) - 1)} \begin{cases} \psi_1 \sqrt{|T_n|} & \text{if \(d < 2\alpha\)} \\ \psi_2 \log_2|T_n|\sqrt{|T_n|} & \text{if \(d = 2\alpha\)} \\ \psi_1 |T_n|^{1-\frac{\alpha}{d}} & \text{if \(d > 2\alpha\)}\end{cases} 
     \label{eq:regret_approx_pruning}
\end{multline}

\paragraph{Step 2: Upper-bounding the pruning estimation error \(R_1\).}
We aim at bounding the estimation error \(R_1\) due to the error incurred by sequentially learning the best pruned tree prediction and the best root node in $\Gamma$ inside each pruned leaves. Note that at each time \(t\), only a subset of nodes of \(\T_0\) are active and output predictions: for any time \(t \geq 1\), let us denote \(\N_t \subset \N(\T_0)\) the set of active nodes (i.e. making a prediction) at time $t$. Remark that \begin{align} \hat f_t(x_t) &= \sum_{n \in \N(\T_0)} \sum_{k=1}^K w_{n,k,t}\hat f_{n,k,t}(x_t) \label{eq:deffhat}\\ &=\sum_{n \in \N_t} \sum_{k=1}^K \tilde w_{n,k,t}\hat f_{n',k',t}(x_t) + \sum_{n \not \in \N_t} \sum_{k=1}^K  \tilde w_{n,k,t} \hat f_t(x_t),
\label{eq:trick_prediction}
\end{align}
by definition of \(\hat f_t\) and the so called trick of prediction with sleeping experts, e.g. in \cite{gaillard2014second}. Recall that \(\tilde \g_t = \nabla_{\tilde \w_t}\ell_t\big(\textstyle\sum_{n \in \N_t} \sum_{k=1}^K \tilde w_{n,k,t} \hat f_{n,k,t}(x_t) + \sum_{n \not \in \N_t} \sum_{k=1}^K \tilde w_{n,k,t} \hat f_t(x_t)\big) \in \R^{|\N(\T_0)|\times K}\), for all \(t \geq 1\). Then, for all \(n \in \N(\T_0), k \in [K]\), \begin{equation}
\tilde g_{n,k,t} = \begin{cases}
\ell_t'(\hat f_t(x_t))\hat f_{n,k,t}(x_t) & \text{if } n \in \N_t, \\
\ell_t'(\hat f_t(x_t))\hat f_t(x_t) & \text{if } n \not \in \N_t.
\end{cases}
\label{eq:grad_cases}
\end{equation}
For any \(t \geq 1, n \in \L(\T)\) and $k \in [K]$, one has:
\begin{align}
    \tilde \g_t^\top \w_t - \tilde g_{n,k,t} &= \sum_{n' \in \N(\T_0)}  \sum_{k'=1}^K w_{n',k',t}\tilde g_{n',k',t} - \tilde g_{n,k,t} & \notag \\
    &= \ell_t'(\hat f_t(x_t)) \underbrace{\sum_{n' \in \N_t} \sum_{k'=1}^K w_{n',k',t} \hat f_{n',k',t}(x_t)}_{= \hat f_t(x_t)} - \tilde g_{n,k,t} & \notag \\
     &= \ell_t'(\hat f_t(x_t)) \Big(\sum_{n' \in \N_t} \sum_{k' =1}^K \tilde w_{n',k',t}\hat f_{n',k',t}(x_t) + \sum_{n' \not \in \N_t} \sum_{k'=1}^K \tilde w_{n',k',t} \hat f_t(x_t)\Big)  - \tilde g_{n,k,t} & \notag \\
    &= \ell_t'(\hat f_t(x_t)) \notag \\
    & \quad \times \begin{cases}
         (\hat f_t(x_t) - \hat f_t(x_t))  \text{ if } n \not \in \N_t\,, \\
        \big(\sum_{n' \in \N_t} \sum_{k' =1}^K \tilde w_{n',k',t}\hat f_{n',k',t}(x_t) + \sum_{n' \not \in \N_t} \sum_{k'=1}^K \tilde w_{n',k',t} \hat f_t(x_t) - \hat f_{n,k,t}(x_t)\big) \text{ else}\,,
    \end{cases} & \notag \\
    &= \begin{cases}
         0 & \text{if } n \not \in \N_t \,,\\
        \tilde \g_t^\top \tilde \w_t - \tilde g_{n,k,t} & \text{else}\,,
        \end{cases} &\notag \\
    &= (\tilde\g_t^\top \tilde \w_t - \tilde g_{n,k,t})\ind{x_t \in \X_n}\,, & \label{eq:equality_sleeping}
\end{align} 
where the second equality follows from \eqref{eq:deffhat}, the third from \eqref{eq:trick_prediction}, and the fourth from \eqref{eq:grad_cases}.
Finally, we obtain
\begin{align}
        (\ell_t(\hat f_t(x_t)) - \ell_t(\hat f_{n,k,t}(x_t)))\ind{x_t \in \X_n} 
        &\le  \ell_t'(\hat f_t(x_t))
         (\hat f_t(x_t) - \hat f_{n,k,t}(x_t))\ind{x_t \in \X_n} & \leftarrow \text{by convexity of } \ell_t \notag \\
        & = (\tilde \g_t^\top \tilde \w_t - \tilde g_{n,k,t})\ind{x_t \in \X_n}  & \notag \\
        &= \tilde \g_t^\top \w_t - \tilde g_{n,k,t} & \leftarrow \text{by } \eqref{eq:equality_sleeping}, 
        \label{eq:convex}
\end{align}
and setting \(T_n = \{1 \leq t \leq T : x_t \in \X_n\}, n \in \L(\T)\):
\begin{align}
    R_1 &= \sumT \sum_{n \in \L(\T)} (\ell_t(\hat f_t(x_t)) - \ell_t(\hat f_{n,k_n,t}(x_t))\ind{x_t \in \X_n} & \leftarrow \{\X_n, n \in \L(\T)\} \text{ partition of } \X \notag \\
    &\leq \sum_{n \in \L(\T)} \sumT (\tilde \g_t ^\top \w_t - \tilde g_{n,k_n,t}) & \leftarrow \text{by } \eqref{eq:convex} \notag \\
    &\leq \sum_{n \in \L(\T)}\Big( C_3 \sqrt{\log\big(K|\N(\T_0)|\big)} \sqrt{\sumT \big(\tilde \g_t^\top \w_t - \tilde g_{n,k_n,t})^2} + C_4G \Big)& \leftarrow \text{by Assumption } \ref{assumption:second_order_algo} \notag \\
    &= C_4G|\L(\T)| + C_3 \sqrt{\log\big(K|\N(\T_0)|\big)} \sum_{n \in \L(\T)} \sqrt{\sum_{t \in T_n} \big(\tilde \g_t^\top \w_t - \tilde g_{n,k_n,t})^2}, & \label{eq:bound_R1}
\end{align}
where last equality holds because for any \(n \in \L(\T), \tilde \g_t^\top \w_t - \tilde g_{n,k_n,t} = 0\) if \(x_t \not \in \X_n\).

\newpage 
\emph{$\bullet$ Case 1: \((\ell_t)_{1 \leq t \leq T}\) convex.} 

Since \(\|\tilde \g_t\|_\infty \leq G, \|\w_t\|_\infty \leq 1, t \in [T]\) by Assumption \ref{assumption:second_order_algo} and using Cauchy-Schwartz inequality we get from Equation \eqref{eq:bound_R1}:
\begin{align}
    R_1 &\leq C_4G|\L(\T)| + 2C_3 \sqrt{\log\big(K|\N(\T_0)|\big)} G\sum_{n \in \L(\T)} \sqrt{|T_n|} \notag\\
    &\leq C_4G|\L(\T)| + 2C_3G \sqrt{\log\big(K|\N(\T_0)|\big)|\L(\T)| \sum_{n \in \L(\T)} |T_n|} \notag \\
    &=C_4G|\L(\T)| + 2C_3G \sqrt{\log\big(K|\N(\T_0)|\big)|\L(\T)| T}.
    \label{eq:regret_pruning_convex_loss}
\end{align} 

In case of convex losses, we finally have by \eqref{eq:decomposition_reg}, \eqref{eq:regret_approx_pruning} and \eqref{eq:regret_pruning_convex_loss} :
\begin{multline*} \Reg_T(f) \leq 2C_3G\sqrt{\log\big(K|\N(\T_0)|\big)|\L(\T)|T} + \left(C_2 \frac{\varepsilon}{2} +C_4\right) G|\L(\T)| + \frac{\varepsilon}{2} G  C_1\sqrt{|\L(\T)|T}  \\  + G|\X|^\alpha \sum_{n \in \L(\T)} L_n(f) 2^{-\alpha(\depth(n)-1)} \begin{cases}
        \psi_1 \sqrt{|T_n|} & \text{if \(d < 2\alpha\)}\,, \\ \psi_2 \log_2|T_n|\sqrt{|T_n|} & \text{if \(d = 2\alpha\)}\,, \\ \psi_1 |T_n|^{1-\frac{\alpha}{d}} & \text{if \(d > 2\alpha\)}\,,
\end{cases}
\end{multline*}
with \(\psi_1, \psi_2\) defined in \eqref{eq:regret_approx_pruning}.
Taking $\varepsilon = T^{-\frac{1}{2}}, K = \lfloor 2BT^{\frac{1}{2}} \rfloor \leq 2BT$, we get:
\begin{multline*} \Reg_T(f) \leq 2C_3G\sqrt{\log\big(2BT|\N(\T_0)|)}\sqrt{\L(\T)|T} + (2^{-1}C_1 + C_2 2^{-1} T^{-\frac{1}{2}} + C_4) G|\L(\T)| \\ + G|\X|^\alpha \sum_{n \in \L(\T)} L_n(f) 2^{-\alpha(\depth(n)-1)} \begin{cases}
        \psi_1 \sqrt{|T_n|} & \text{if \(d < 2\alpha\)}\,, \\ \psi_2 \log_2|T_n|\sqrt{|T_n|} & \text{if \(d = 2\alpha\)}\,, \\ \psi_1 |T_n|^{1-\frac{\alpha}{d}} & \text{if \(d > 2\alpha\)}\,,
\end{cases}
\end{multline*}

Since this inequality holds for all pruning \(\T \in \Pcal(\T_0)\), one can take the infimum over all pruning in \(\Pcal(\T_0)\) to get the desired upper-bound:
\begin{multline*}
    \Reg_T(f) \leq \inf_{\T \in \Pcal(\T_0)} \Bigg\{\beta_1\sqrt{\L(\T)|T} + \beta_2 |\L(\T)|  \\ + G|\X|^\alpha \sum_{n \in \L(\T)} L_n(f) 2^{-\alpha(\depth(n)-1)} \begin{cases}
        \psi_1 \sqrt{|T_n|} & \text{if \(d < 2\alpha\)}\,, \\ \psi_2 \log_2|T_n|\sqrt{|T_n|} & \text{if \(d = 2\alpha\)}\,, \\ \psi_1 |T_n|^{1-\frac{\alpha}{d}} & \text{if \(d > 2\alpha\)}\,,
\end{cases}\Bigg\},
\end{multline*}
with $\beta_1 = 2C_3G\sqrt{\log\big(2BT|\N(\T_0)|)}$ and $\beta_2 = G(2^{-1}C_1 + C_2 2^{-1} T^{-\frac{1}{2}} + C_4)$.

\newpage 

\emph{$\bullet$ Case 2: \((\ell_t)_{1 \leq t \leq T}\) \(\eta\)-exp-concave.} 

If the sequence of loss functions \((\ell_t)\) is \(\eta\)-exp-concave for some \(\eta >0\), then thanks to a Lemma in \cite{hazan2016introduction} we have for any \(0 < \mu \leq \frac{1}{2}\min \{\frac{1}{G}, \eta\}\) and all \(t \geq 1, n \in \L(\T), k \in [K]\):
\begin{align}
(\ell_t(\hat f_t(x_t)) - \ell_t(\hat f_{n,k,t}(x_t)))\ind{x_t \in \X_n} &\leq \big(\tilde \g_t^\top \tilde \w_t - \tilde g_{n,k,t} - \frac{\mu}{2} \big(\tilde \g_t^\top  \tilde \w_t - \tilde g_{n,k,t}\big)^2\big)\ind{x_t \in \X_n} &\notag \\
&=  \tilde \g_t^\top \w_t - \tilde g_{n,k,t} - \frac{\mu}{2} \big(\tilde \g_t^\top \w_t - \tilde g_{n,k,t}\big)^2 & \leftarrow \text{by } \eqref{eq:equality_sleeping}\,. \label{eq:exp_concave}
\end{align}
Summing \eqref{eq:exp_concave} over $t \in [T]$ and \(n \in \L(\T)\), we get:
\begin{align}
    R_1 &\leq  \sum_{n \in \L(\T)} \sum_{t \in T_n} \tilde \g_t ^\top \tilde \w_t - \tilde g_{n,k,t} - \frac{\mu}{2} \sum_{n \in \N(\Pcal)} \sum_{t \in T_n} \big(\tilde \g_t^\top  \w_t - \tilde g_{n,k,t}\big)^2 \notag \\
    &\leq C_4G |\L(\T)| + \tilde C_3 \sum_{n \in \L(\T)} \sqrt{\sum_{t \in T_n} \big(\tilde \g_t^\top \w_t - \tilde g_{n,k,t})^2} - \frac{\mu}{2} \sum_{n \in \L(\T)} \sum_{t \in T_n} \big(\tilde \g_t^\top  \w_t - \tilde g_{n,k,t}\big)^2 & \leftarrow \text{by } \eqref{eq:bound_R1}\,, \label{eq:before_young}
\end{align}
where we set $\tilde C_3 = C_3 \sqrt{\log\big(K|\N(\T_0)|\big)}$. 
Young's inequality gives, for any \(\nu >0\), the following upper-bound:
\begin{equation}
    \sqrt{\sum_{t\in T_n} \big(\tilde \g_t^\top \w_t - \tilde g_{n,k,t})^2} \leq \frac{1}{2\nu} + \frac{\nu}{2} \sum_{t \in T_n} \big(\tilde \g_t^\top \w_t - \tilde g_{n,k,t})^2 \,.
    \label{eq:young}
\end{equation}
Finally, plugging \eqref{eq:young} with \(\nu = \mu/\tilde C_3 > 0\) in \eqref{eq:before_young}, we get
\begin{align}
    R_1 &\leq C_4G |\L(\T)| + \tilde C_3 \sum_{n \in \L(\T)} \left(\frac{\tilde C_3}{2\mu} + \frac{\mu}{2\tilde C_3}\sum_{t \in T_n} \big(\tilde \g_t^\top \w_t - \tilde g_{n,k,t})^2\right) - \frac{\mu}{2} \sum_{n \in \L(T)} \sum_{t \in T_n} \big(\tilde \g_t^\top  \w_t - \tilde g_{n,k,t}\big)^2 \notag \\
    &= \left(\frac{C_3^2 \log\big(K|\N(\T_0)|\big)}{2\mu} + C_4G\right)|\L(\T)|.
    \label{eq:regret_pruning_exp_concave_loss}
\end{align}

To conclude, if \((\ell_t)\) are \(\eta\)-exp-concave, one has via \eqref{eq:decomposition_reg}, \eqref{eq:regret_approx_pruning} and \eqref{eq:regret_pruning_exp_concave_loss}
\begin{equation*} \Reg_T(f) \leq \beta_3|\L(\T)| + G|\X|^\alpha \sum_{n \in \L(\T)} L_n(f) 2^{-\alpha(\depth(n)-1)} \begin{cases}
        \psi_1 \sqrt{|T_n|} & \text{if \(d < 2\alpha\)}\,, \\ \psi_2 \log_2|T_n|\sqrt{|T_n|} & \text{if \(d = 2\alpha\)}\,, \\ \psi_1 |T_n|^{1-\frac{\alpha}{d}} & \text{if \(d > 2\alpha\)}\,,
\end{cases}
\end{equation*}
with $\beta_3 = \frac{C_3^2 \log\big(2 BT|\N(\T_0)|\big)}{2\mu} + C_4G + 2^{-1}G (C_1 + C_2T^{- \frac{1}{2}})$, \(0<\mu<\frac{1}{2}\min\{\frac{1}{G}, \eta\}\) and \(\psi_1, \psi_2\) defined in \eqref{eq:regret_approx_pruning}.
Again, taking infimum over \(\T \in \Pcal(\T_0)\) gives the result.

\paragraph{Worst case regret bound} Note that since we assume that $\|f\|_\infty \leq B$, and that all local predictors $\hat f_{n,k}, n \in \N(\T_0), k \in [K]$ in Algorithm \ref{alg:Local_Adapt_Algo} are clipped in $[-B,B]$, we first have for any $x \in \X$, \[ |\hat f_t(x)| = \sum_{n\in \N(\T_0)} \sum_{k=1}^K w_{n,k,t}|\hat f_{n,k,t}(x)| \leq B \sum_{n\in \N(\T_0)} \sum_{k=1}^K w_{n,k,t} = B.\]
 
Thus,
 \begin{align}
     \Reg_T(f) &= \sumT \ell_t(\hat f_t(x_t)) - \ell_t(f^*(x_t)) & \notag\\
     &\leq \sumT G |\hat f_{t}(x_t) - f^*(x_t)| & \leftarrow \text{$\ell_t$ is $G$-Lipschitz} \notag \\
     &\leq G \sumT (|\hat f_{t}(x_t)| + |f^*(x_t)|) &\notag \\
     &= 2BGT & \label{eq:worst_case}
 \end{align}
 
\end{proof}

\newpage 
\section{Proof of Corollary \ref{cor:regret_flat_pruning_exp_concave}}
\label{appendix:proof_corollary:flat_pruning}
We state here a complete version of Corollary \ref{cor:regret_flat_pruning_exp_concave}.

\begin{cor}
\label{cor:complete_version_regret_flat_pruning_exp_concave}
    Let $\alpha \in (0,1], 1 \leq d \leq 2\alpha$. Under the same assumptions as in Theorem \ref{theorem:optimal_local_regret}, Algorithm~\ref{alg:Local_Adapt_Algo} achieves a regret with respect to any $f \in \C^\alpha(\X,L), L > 0$:
    \[
         \Reg_T(f) \lesssim
        \inf_{\T \in\Pcal(\T_0)} \left\{ \sum_{n \in \L(\T)}
            \min \left(1 + L_n(f)|\X_n|^\alpha, \Big(L_n(f) |\X_n|^\alpha \Big)^{\frac{1}{2\alpha}} \right) \sqrt{T_n} \right\} \,,
    \]
    where $\lesssim$ is a rough inequality that depends on $C_i$, $i=1,\ldots,4$ but is independent of $L,X,T$.
    Moreover, if $(\ell_t)$ are exp-concave:
    \[
         \Reg_T(f) \lesssim
        \inf_{\T \in\Pcal(\T_0)} \left\{ \sum_{n \in \L(\T)}
             \min \left( L_n(f)|\X_n|^\alpha \sqrt{|T_n|}, \big(L_n(f)|\X_n|^\alpha \big)^{\frac{2}{2\alpha + 1}} |T_n|^{\frac{1}{2\alpha + 1}}  \right) \right\} \,,
    \]
where $\lesssim$ also depends on the exp-concavity constant.
\end{cor}

\begin{proof}[{\bfseries of Corollary \ref{cor:complete_version_regret_flat_pruning_exp_concave}}]

We consider 2 cases: 
\begin{enumerate}
    \item \emph{Case $d < 2\alpha$} (i.e. $d=1,  \alpha \in (\frac{1}{2},1]$).

Let \(f \in \C^\alpha(\X,L)\) and \(L >0\) and fix any pruning $\T \in \Pcal(\T_0)$. We will apply Theorem~\ref{theorem:optimal_local_regret} to an extended pruning $\T'$, in which we extend each leaf $n \in \L(\T)$ by a regular tree of depth $h_n \in \mathbb N$ to be optimized later in the proof. In particular, for each leaf $n$ in the original pruning $\T$, $\T'$ has $2^{h_n}$ leaves $m$ at depth $\depth(m) = \depth(n) + h_n \geq \depth(n)$ with $L_m(f) \leq L_n(f)$. In particular, when $h_n = 0$, the original pruning $\T$ is recovered.

\begin{enumerate}

\item \emph{ Case \((\ell_t)\) convex:}

Thanks to Theorem \ref{theorem:optimal_local_regret} (without applying Inequality~\eqref{eq:regret_pruning_convex_loss} in the term depending on $C_3$), one has for \(d = 1 < 2\alpha\): 
\begin{align}
    \Reg_T(f) 
        &\leq 2C_3 \sqrt{\log\big(K|\N(\T_0)|\big)} \sum_{m \in \L(\T')} \sqrt{|T_m|} + C_4G|\L(\T_1)| \nonumber \\
        & \quad + G\psi_1  \sum_{m \in \L(\T')} L_m(f)  |\X_m|^\alpha  \sqrt{|T_m|}, \nonumber \\
        &\leq \min_{h_n \in \mathbb N} \left\{ C \sum_{n \in \L(\T)} \Big(\sqrt{2^{h_n} |T_n|} + 2^{h_n} + L_n(f) |\X_n|^\alpha 2^{-\alpha h_n} \sqrt{2^{h_n} |T_n|}\Big) \right\}
    \label{eq:all_pruning_bound_convex}
\end{align}
where $C >0$ is some constant that depends on $C_3,C_4,G,X, \log(T), B$ and $\psi_1$ (defined in Theorem~\ref{theorem:optimal_local_regret}) but independent of other quantities $L_n(f), T, T_n$, that is used to simplify the presentation and may change from a display to another along the proof. Then, optimizing over $h_n$ so that
\[
    \sqrt{ 2^{h_n}|T_n| } =  L_n(f)  |\X_n|^\alpha 2^{-\alpha h_n} \sqrt{2^{h_n} |T_n|} \,,
\]
we set
\[
    h_n = \max\left\{ 0, \frac{1}{\alpha} \log_2 \left(L_n(f)|\X_n|^\alpha \right) \right\} \geq 0
\]
which yields
\[
\Reg_T(f) \leq C \sum_{n \in \L(\T)} \min\left\{1 + L_n(f)|\X_n|^\alpha, \Big(L_n(f) |\X_n|^\alpha \Big)^{\frac{1}{2\alpha}}\right\} \sqrt{|T_n|}  \,. 
\]

\item \emph{Case \((\ell_t)\) exp-concave:}

Since \((\ell_t)\) are exp-concave, Theorem \ref{theorem:optimal_local_regret} (with Inequality \eqref{eq:regret_pruning_exp_concave_loss}) gives, for \(d < 2\alpha\), for any extension $\T_1$ of pruning $\T \in \Pcal(\T_0)$:

\begin{align}
    \Reg_T(f)
        & \leq C\left(|\L(\T_1)| + \sum_{m \in \L(\T_1)} L_m(f)|\X_m|^\alpha \sqrt{|T_m|}\right), \\
        &\leq  C \sum_{n \in \L(\T)} \Big(2^{h_n} + L_n(f) |\X_n|^\alpha 2^{-\alpha h_n} \sqrt{2^{h_n}|T_n|}\Big) \,,
\label{eq:all_pruning_bound}
\end{align}
where again $C >0$ is a constant independent of $L$, $L_n(f)$, $|T_n|$ and $h_n$ that may change from a display to another. Optimizing over $h_n$ by equalizing the terms:
\[
    2^{h_n} = L_n(f) |\X_n|^\alpha 2^{-\alpha h_n} \sqrt{2^{h_n}|T_n|}
\]
leads to
\[
    h_n = \max\left\{0, \frac{2}{2\alpha + 1} \log_2 \big(L_n(f) |\X_n|^{\alpha} |T_n|^{\frac{1}{2}}\big) \right\} \,,
\]
which yields and concludes the proof:
\begin{equation*}
    \Reg_T(f) \leq C \sum_{n \in \L(\T)}  \min \left\{ L_n(f)|\X_n|^\alpha \sqrt{|T_n|}, \big(L_n(f)|\X_n|^\alpha \big)^{\frac{2}{2\alpha + 1}} |T_n|^{\frac{1}{2\alpha + 1}}  \right\}
\end{equation*}

\end{enumerate} 

\item \emph{Case $d = 2\alpha$.}

The proof is the same as for the case $d < 2\alpha$ but with $C$ now depending on $\psi_2$ (also defined in Theorem \ref{theorem:minimax_CT_param_free}) rather than $\psi_1$. We get the same result.




\end{enumerate} 
\end{proof}

\newpage
\section{Proof of Equation \eqref{eq:simple_pruning_bound}}

\label{appendix:proof_simplified_bound}

One has, for any pruning $\T \in \Pcal(\T_0)$ and some $f \in \C^\alpha(\X,L)$:
\begin{align*} 
\Reg_T(f) \lesssim \sum_{n \in \L(\T)} \big( L_n(f) |\X_n|^\alpha \big)^{\frac{2}{2\alpha + 1}} |T_n|^{\frac{1}{2\alpha + 1}} &= \sum_{n \in \L(\T)} \big(L_n(f)^{\frac{1}{\alpha}}|\X_n| \big)^{\frac{2\alpha}{2\alpha + 1}} |T_n|^{\frac{1}{2\alpha + 1}} \\
& \leq  \Big(\sum_{n \in \L(\T)} L_n(f)^{\frac{1}{\alpha}} |\X_n|\Big)^{\frac{2\alpha}{2\alpha + 1}} \Big|\sum_{n \in \L(\T)} T_n\Big|^{\frac{1}{2\alpha + 1}} \\
&= \Big(\sum_{n \in \L(\T)} L_n(f)^{\frac{1}{\alpha}} |\X_n|\Big)^{\frac{2\alpha}{2\alpha + 1}} |T|^{\frac{1}{2\alpha + 1}}
\end{align*}
where inequality is obtained with Hölder's inequality with $p = (2\alpha + 1)/2\alpha$ and $q = 2\alpha +1$.

One could also write: \[\Big(\sum_{n \in \L(\T)} L_n(f)^{\frac{1}{\alpha}} |\X_n|\Big)^{\frac{2\alpha}{2\alpha + 1}} = \left(|\X| \sum_{n \in \L(\T)} L_n(f)^{\frac{1}{\alpha}} \frac{|\X_n|}{|\X|}\right)^{\frac{2\alpha}{2\alpha + 1}} := \left(|\X|^\alpha \| f\|_{\L(\T),\frac{1}{\alpha}}\right)^{\frac{2}{2\alpha +1}},\]
where $f \mapsto \|f\|_{\L(\T),\frac{1}{\alpha}}$ is some $\frac{1}{\alpha}$-norm (or expectation) of $f$ over leaves $n \in \L(\T)$ with probability $|\X_n|/|\X| = 2^{-\depth(n)}, n \in \L(\T)$.






\newpage
\section{Comparison with \cite{kuzborskij2020locally}}

\label{appendix:comparison_with_CB}

 Let $f \in \C^\alpha(\X,L), (M^{(k)})_{1\leq k\leq \depth(\T_0)}$ such that $M^{(k)} \geq L_n(f)$ for any $n \in \N(\T_0), \depth(n) = k$ and $T^{(k)} = \{1 \leq t \leq T : x_t \in \X_n, \depth(n) = k\}$. Let $\T$ be any pruning and let $\alpha = 1, d = 1$. We have for the squared (exp-concave) loss, according to Corollary \ref{cor:regret_flat_pruning_exp_concave}:
\begin{align}
    \Reg_T(f) &\lesssim \sum_{n \in \L(\T)} \min\left\{ \left(L_n(f) \frac{|\X_n|}{|\X|}\right)^{\frac{2}{3}} |T_n|^{\frac{1}{3}}, \left(L_n(f)\frac{|\X_n|}{|\X|}\right)^{\frac{1}{2}}|T_n|^{\frac{1}{2}}\right\} & \notag \\
    &= \sum_{k=1}^{\depth(\T)} \sum_{n \in \L(\T):\depth(n)=k} \min\left\{ \left(L_n(f) \frac{|\X_n|}{|\X|}\right)^{\frac{2}{3}} |T_n|^{\frac{1}{3}}, \left(L_n(f)\frac{|\X_n|}{|\X|}\right)^{\frac{1}{2}}|T_n|^{\frac{1}{2}}\right\} & \notag\\
    &\leq \sum_{k=1}^{\depth(\T)} \min\left\{ \left(\sum_{n \in \L(\T):\depth(n)=k} L_n(f) 2^{-k}\right)^{\frac{2}{3}} |T^{(k)}|^{\frac{1}{3}},  \left(\sum_{n \in \L(\T):\depth(n)=k}  L_n(f)2^{-k}\right)^{\frac{1}{2}}|T^{(k)}|^{\frac{1}{2}}\right\} & \label{eq:holder1} \\
    &\leq \sum_{k=1}^{\depth(\T)} \min\left\{ \left(M^{(k)}w^{(k)}\right)^{\frac{2}{3}} |T^{(k)}|^{\frac{1}{3}},  \left(M^{(k)}w^{(k)}\right)^{\frac{1}{2}} |T^{(k)}|^{\frac{1}{2}}\right\} & \notag \\
    &\leq \min\left\{\left(\bar M(f)\right)^{\frac{2}{3}}T^{\frac{1}{3}}, \left(\bar M(f)T\right)^{\frac{1}{2}}\right\} & \label{eq:holder2}
\end{align}
where $\bar M(f) = \sum_{k=1}^{\depth(\T)} w^{(k)}M^{(k)}$ with $w^{(k)} = \sum_{n \in \L(\T) : \depth(n) = k} \frac{|\X_n|}{|\X|} = 2^{-k}|\{n \in \L(\T) : \depth(n) = k\}|$ the proportion of leaves in $\L(\T)$ at level $k$ in $\N(\T_0)$ and where we applied Hölder's inequality to get \eqref{eq:holder1} and \eqref{eq:holder2}. The last upper-bound  recovers and improves the one of \citet{kuzborskij2020locally} for dimension $d=1$, as described in the main part of the paper. For higher dimension $d \geq 2$, both for exp-concave and convex losses, Theorem \ref{theorem:optimal_local_regret} gives for any pruning $\T$ and any function $f \in \C^1(\X,L)$  (ignoring the dependence on $\X$):
\begin{align*} \Reg_T(f) \lesssim \sum_{n \in \L(\T)} L_n(f) |T_n|^{1-\frac{1}{d}} &= \sum_{k=1}^{\depth(\T)} \sum_{n \in \L(\T):\depth(n)=k} L_n(f) |T_n|^{\frac{d-1}{d}} \\ &\leq \sum_{k=1}^{\depth(\T)} M^{(k)} |\{n :\depth(n)=k\}|^{\frac{1}{d}} |T^{(k)}|^{\frac{d-1}{d}},
\end{align*}
which grows as $O\big(\sum_k M^{(k)}|T^{(k)}|^{\frac{d-1}{d}}\big)$ compared to $O\big(\sum_k (M^{(k)}|T^{(k)}|)^{\frac{d}{d+1}}\big)$ in \cite{kuzborskij2020locally}. As a consequence, if for every level $k = 1, \dots, \depth(\T)$, $M^{(k)} |T^{(k)}|^{\frac{d-1}{d}} \geq (M^{(k)}|T^{(k)}|)^{\frac{d}{d+1}}$ it turns out that $M^{(k)} \geq |T^{(k)}|^{\frac{1}{d}}$ which leads to an equivalent bound $O(\sum_k |T^{(k)}|) = O(T)$ which corresponds to the worst case regret bound \eqref{eq:worst_case}. As a conclusion, our bound recovers and improves their results in particular with a dependence to lower constants $L_n(f)$ and lower time rate $|T_n|^{1-\frac 1 d}$.

\newpage 
\section{Experiments}
\label{appendix:xp}

The following presents experimental results in a synthetic regression setting for both the Chaining Tree method (Algorithm \ref{alg:training_CT}) and the Locally Adaptive Online Regression (Algorithm \ref{alg:Local_Adapt_Algo}. We consider the model \( y_t = f(x_t) + \varepsilon_t \), where \( \varepsilon_t \sim \mathcal{N}(0, \sigma^2) \) with \( \sigma = 0.5, f(x) = \sin(10x) + \cos(5x) + 5 \), for \( x \in \mathcal{X} = [0, 1] \) and \( \sup_x |f'(x)| \leq 15 =: L\). Furthermore, we assume that \( x_t \) is independently drawn from the uniform distribution \( \mathcal{U}(\mathcal{X})\).

Refer to Theorem \ref{theorem:minimax_CT_param_free}, Theorem \ref{theorem:optimal_local_regret} and Corollary \ref{cor:regret_flat_pruning_exp_concave} in the paper to compare experimental results to theoretical guarantees. Figures can be reproduced using the code available on \href{https://github.com/paulliautaud/Minimax-Locally-Adaptive-Online-Regression}{GitHub}.

\textit{Key observations:}
\begin{itemize}[nosep,topsep=-\parskip]
    \item For the squared loss, $\ell_t(\hat y) = (\hat y - y_t)^2$, Algorithm \ref{alg:Local_Adapt_Algo} achieves a time rate of $O(T^{\frac{1}{3}})$ compared to $O(\sqrt{T})$ for Algorithm \ref{alg:training_CT} - see Figure \ref{fig:xp_square_loss}. However, the trade-off is an increased dependence on the smoothness $L$, shifting from $O(L^{\frac{1}{2}})$ to $O(L^{\frac{2}{3}})$;
    \item We observe in Figure \ref{fig:xp_lip_combined} that Algorithm \ref{alg:Local_Adapt_Algo} reduces regret with respect to $L$: it achieves $O(\sqrt{L})$ for absolute loss in Fig. \ref{fig:xp_lip_abs} and $O(L^{\frac{2}{3}})$ for square loss in Fig. \ref{fig:xp_lip_square};
    \item In Figure \ref{fig:xp_lip_square} and Figure \ref{fig:xp_lip_abs}, we observe that both the experimental and theoretical curves level off once \(L\) increases beyond a certain threshold $L_0 \gtrsim BT$. Indeed, we demonstrated that for any \(f \in \C^\alpha(\X,L)\),
\[
\Reg_T(f) \lesssim \min\bigg\{BT, \sum_n L_n(f)|T_n|^{1 - \frac{1}{d}}\bigg\}.
\]
See  Appendix \ref{appendix:proof_theorem:optimal_local_regret} and Equation \eqref{eq:worst_case} for more details.
\end{itemize}

\begin{figure}[htbp]
    \centering
    \subfigure[Regret with $\ell_t(\hat y) = (\hat y - y_t)^2$]{
        \input{reg_square_loss}
        \label{fig:xp_square_loss}
    }
    \hfill 
    \subfigure[Regret with $\ell_t(\hat y) = |\hat y - y_t|$]{
        \input{reg_absolute_loss}
        \label{fig:xp_absolute_loss}
    }
    \caption{Comparison of regret as a function of \(T\) for square and absolute loss functions. The dotted lines represent the theoretical results (where $O$ hides terms in $\log T$), while the solid lines show the actual performance of our algorithms (mean $\pm$ std over $5$ runs).}
    \label{fig:regret_comparison}
\end{figure}

\newpage

\begin{figure}[htbp]
    \centering
    \subfigure[Regret with $\ell_t(\hat y) = (\hat y - y_t)^2$]{
        \begin{tikzpicture}[scale = 0.85]

\definecolor{darkgray176}{RGB}{176,176,176}
\definecolor{green}{RGB}{0,128,0}
\definecolor{green01270}{RGB}{0,127,0}
\definecolor{lightgray204}{RGB}{204,204,204}

\begin{axis}[
legend cell align={left},
legend style={
  fill opacity=0.8,
  draw opacity=1,
  text opacity=1,
  at={(0.97,0.03)},
  anchor=south east,
  draw=none
},
log basis y={10},
tick align=outside,
tick pos=left,
x grid style={darkgray176},
xlabel={\(\displaystyle L\)},
xmin=-0.95, xmax=19.95,
xtick style={color=black},
xtick={0,1,2,3,4,5,6,7,8,9,10,11,12,13,14,15,16,17,18,19},
xticklabel style={
    rotate=45,               
    font=\small,             
    yshift=0.5ex             
},
xticklabels={
  $2.3 \times 10^{-1}$,
  ,
  $3.3 \times 10^{1}$,
  ,
  $6.6 \times 10^{1}$,
  ,
  $9.9 \times 10^{1}$,
  ,
  $1.3 \times 10^{2}$,
  ,
  $1.7 \times 10^{2}$,
  ,
  $2 \times 10^{2}$,
  ,
  $2.3 \times 10^{2}$,
  ,
  $2.6 \times 10^{2}$,
  ,
  $3 \times 10^{2}$,
  ,
  $3.3 \times 10^{2}$,
  ,
  $3.6 \times 10^{2}$,
  ,
  $4 \times 10^{2}$,
  ,
  $4.3 \times 10^{2}$,
  ,
  $4.6 \times 10^{2}$,
  ,
},
y grid style={darkgray176},
ylabel={\(\displaystyle \operatorname{Reg}_T(f)\)},
ymin=81.4836947858042, ymax=20088.8460922775,
ymode=log,
ytick style={color=black},
ytick={1,10,100,1000,10000,100000,1000000},
yticklabels={
  \(\displaystyle {10^{0}}\),
  \(\displaystyle {10^{1}}\),
  \(\displaystyle {10^{2}}\),
  \(\displaystyle {10^{3}}\),
  \(\displaystyle {10^{4}}\),
  \(\displaystyle {10^{5}}\),
  \(\displaystyle {10^{6}}\)
}
]
\path [draw=green, fill=green, opacity=0.2]
(axis cs:0,1401.02173523303)
--(axis cs:0,1349.85665463914)
--(axis cs:1,1721.0037051897)
--(axis cs:2,2146.65304909307)
--(axis cs:3,2562.78180806498)
--(axis cs:4,2808.00997424685)
--(axis cs:5,3128.11054619216)
--(axis cs:6,3213.0802544429)
--(axis cs:7,3190.2716704239)
--(axis cs:8,3217.60516606567)
--(axis cs:9,3253.90981898539)
--(axis cs:10,3309.16817502454)
--(axis cs:11,3299.12862817825)
--(axis cs:12,3325.87618031563)
--(axis cs:13,3221.21242530759)
--(axis cs:14,3305.86680139052)
--(axis cs:15,3333.27376647312)
--(axis cs:16,3318.37216636622)
--(axis cs:17,3233.4769752122)
--(axis cs:18,3238.53481860946)
--(axis cs:19,3322.5740532311)
--(axis cs:19,3494.51438494039)
--(axis cs:19,3494.51438494039)
--(axis cs:18,3406.46936702271)
--(axis cs:17,3447.44302887082)
--(axis cs:16,3451.48153040912)
--(axis cs:15,3468.38548460006)
--(axis cs:14,3475.43195875116)
--(axis cs:13,3412.12219129109)
--(axis cs:12,3476.9851742991)
--(axis cs:11,3400.17658697684)
--(axis cs:10,3476.45829231927)
--(axis cs:9,3348.86704569375)
--(axis cs:8,3447.70605321313)
--(axis cs:7,3386.63503319873)
--(axis cs:6,3380.40960537412)
--(axis cs:5,3267.44406299606)
--(axis cs:4,3267.29286625945)
--(axis cs:3,2989.86444723215)
--(axis cs:2,2488.64038391303)
--(axis cs:1,1871.84252451784)
--(axis cs:0,1401.02173523303)
--cycle;

\path [draw=red, fill=red, opacity=0.2]
(axis cs:0,123.502230784003)
--(axis cs:0,104.662890843286)
--(axis cs:1,301.31207672288)
--(axis cs:2,614.512225367162)
--(axis cs:3,857.700695622485)
--(axis cs:4,1008.65935378922)
--(axis cs:5,1269.07985154328)
--(axis cs:6,1385.89147058775)
--(axis cs:7,1498.48637255773)
--(axis cs:8,1413.62501197232)
--(axis cs:9,1712.63927164402)
--(axis cs:10,1816.8691913625)
--(axis cs:11,1842.79987202526)
--(axis cs:12,1872.2635392331)
--(axis cs:13,1862.70390698172)
--(axis cs:14,1963.31750108658)
--(axis cs:15,1945.4662202122)
--(axis cs:16,1915.396287515)
--(axis cs:17,1947.15028785311)
--(axis cs:18,1951.79509503216)
--(axis cs:19,2032.5288909674)
--(axis cs:19,2085.07620917088)
--(axis cs:19,2085.07620917088)
--(axis cs:18,2072.6534964808)
--(axis cs:17,2069.4997693626)
--(axis cs:16,2095.60125780366)
--(axis cs:15,2023.02078376294)
--(axis cs:14,2034.99676878472)
--(axis cs:13,1984.80816336713)
--(axis cs:12,1980.25795637)
--(axis cs:11,1963.80385985715)
--(axis cs:10,1947.15999993443)
--(axis cs:9,1800.65880547748)
--(axis cs:8,1929.85511261043)
--(axis cs:7,1658.5365056486)
--(axis cs:6,1608.12063441739)
--(axis cs:5,1397.78046926416)
--(axis cs:4,1329.56460093773)
--(axis cs:3,1088.71280700811)
--(axis cs:2,702.545555512446)
--(axis cs:1,417.321390888949)
--(axis cs:0,123.502230784003)
--cycle;

\addplot [semithick, green01270, dash pattern=on 5.55pt off 2.4pt]
table {%
0 1228.18233935842
1 4812.32136744587
2 8483.26754965525
3 12137.5363975726
4 14970.8617642846
5 15068.5915702398
6 15024.2303816984
7 15248.9602409548
8 15303.4136824979
9 15271.2452045714
10 15244.9464765681
11 15199.9251096312
12 14932.8854385535
13 15281.2916758713
14 15509.221897086
15 14941.9971792576
16 15630.9575396471
17 15639.864238349
18 15329.6511900994
19 15331.2677574006
};
\addlegendentry{\scriptsize $O\big(L\sqrt{T} \wedge BT \big)$ - Th. \ref{theorem:minimax_CT_param_free}}
\addplot [semithick, red, dash pattern=on 5.55pt off 2.4pt]
table {%
0 532.451608046926
1 2465.37002242276
2 3462.85178254721
3 4226.74737204146
4 5163.67341498628
5 5461.35756957707
6 5902.07959533346
7 6714.65171452375
8 6840.77704658768
9 7413.23432815047
10 7801.08849476368
11 7801.15995204166
12 8295.52853020267
13 8658.68414024804
14 9120.41193839395
15 9959.72011055941
16 9718.43757939153
17 10270.4024113989
18 10260.4965910536
19 10969.056996573
};
\addlegendentry{\scriptsize $O\big(L^{\frac{2}{3}}T^\frac{1}{3} \wedge \sqrt{LT} \wedge BT \big)$ - Cor. \ref{cor:regret_flat_pruning_exp_concave}}
\addplot [semithick, green01270]
table {%
0 1391.26937465527
1 1772.91338431126
2 2192.80644848053
3 2635.05110635475
4 3076.53091590526
5 3145.86877282596
6 3313.43722950528
7 3302.84730325996
8 3300.17837099399
9 3267.52012502288
10 3418.57651107218
11 3370.42443517804
12 3377.73286412557
13 3378.45052120443
14 3450.25824139761
15 3393.07917995571
16 3428.08499860237
17 3304.8669580166
18 3376.27271146581
19 3363.3190780419
};
\addlegendentry{\scriptsize Chaining Tree - Alg. \ref{alg:training_CT}}
\addplot [semithick, red]
table {%
0 115.731485703484
1 358.898997009765
2 671.264342434045
3 951.97764075724
4 1150.45460083763
5 1318.53638232848
6 1449.47973020119
7 1623.06626257376
8 1831.18741674908
9 1773.99677399357
10 1927.70258458375
11 1901.00459563539
12 1913.85842150133
13 1944.94081896628
14 2016.53219293091
15 1982.83141342527
16 2014.79694570892
17 1983.3829201269
18 2017.49163909754
19 2071.4457767697
};
\addlegendentry{\scriptsize Locally Adaptive Online Reg. - Alg. \ref{alg:Local_Adapt_Algo}}
\end{axis}

\end{tikzpicture}
        \label{fig:xp_lip_square}
    }
    \hfill
    \subfigure[Regret with $\ell_t(\hat y) = |\hat y - y_t|$]{
        \begin{tikzpicture}[scale = 0.85]

\definecolor{darkgray176}{RGB}{176,176,176}
\definecolor{green}{RGB}{0,128,0}
\definecolor{green01270}{RGB}{0,127,0}
\definecolor{lightgray204}{RGB}{204,204,204}

\begin{axis}[
legend cell align={left},
legend style={
  fill opacity=0.8,
  draw opacity=1,
  text opacity=1,
  at={(0.97,0.03)},
  anchor=south east,
  draw=none
},
log basis y={10},
tick align=outside,
tick pos=left,
x grid style={darkgray176},
xlabel={\(\displaystyle L\)},
xmin=-1.45, xmax=30.45,
xtick style={color=black},
xtick={0,1,2,3,4,5,6,7,8,9,10,11,12,13,14,15,16,17,18,19,20,21,22,23,24,25,26,27,28,29},
xticklabel style={
    rotate=45,               
    font=\small,             
    yshift=0.5ex             
},
xticklabels={
  $2.3 \times 10^{-1}$,
  ,
  $3.3 \times 10^{1}$,
  ,
  $6.6 \times 10^{1}$,
  ,
  $9.9 \times 10^{1}$,
  ,
  $1.3 \times 10^{2}$,
  ,
  $1.7 \times 10^{2}$,
  ,
  $2 \times 10^{2}$,
  ,
  $2.3 \times 10^{2}$,
  ,
  $2.6 \times 10^{2}$,
  ,
  $3 \times 10^{2}$,
  ,
  $3.3 \times 10^{2}$,
  ,
  $3.6 \times 10^{2}$,
  ,
  $4 \times 10^{2}$,
  ,
  $4.3 \times 10^{2}$,
  ,
  $4.6 \times 10^{2}$,
  ,
},
y grid style={darkgray176},
ylabel={\(\displaystyle \operatorname{Reg}_T(f)\)},
ymin=38.4630715512988, ymax=9170.07359621052,
ymode=log,
ytick style={color=black},
ytick={1,10,100,1000,10000,100000},
yticklabels={
  \(\displaystyle {10^{0}}\),
  \(\displaystyle {10^{1}}\),
  \(\displaystyle {10^{2}}\),
  \(\displaystyle {10^{3}}\),
  \(\displaystyle {10^{4}}\),
  \(\displaystyle {10^{5}}\)
}
]
\path [draw=green, fill=green, opacity=0.2]
(axis cs:0,52.9073221612647)
--(axis cs:0,49.3292332968405)
--(axis cs:1,153.806874986131)
--(axis cs:2,259.417361325582)
--(axis cs:3,378.869670429165)
--(axis cs:4,524.282170021799)
--(axis cs:5,554.202194513371)
--(axis cs:6,704.971937438064)
--(axis cs:7,823.643775644355)
--(axis cs:8,888.448197888621)
--(axis cs:9,816.40004167067)
--(axis cs:10,919.619524938383)
--(axis cs:11,948.818314501921)
--(axis cs:12,942.926047171483)
--(axis cs:13,987.99546803176)
--(axis cs:14,961.828420172007)
--(axis cs:15,967.508988820581)
--(axis cs:16,1010.77349990841)
--(axis cs:17,1009.66389756154)
--(axis cs:18,1019.46044531985)
--(axis cs:19,981.43666164935)
--(axis cs:20,979.741648101916)
--(axis cs:21,982.173688900263)
--(axis cs:22,975.503124699174)
--(axis cs:23,983.667127029819)
--(axis cs:24,961.541141339699)
--(axis cs:25,1013.36952189853)
--(axis cs:26,963.062867385486)
--(axis cs:27,1023.43599342211)
--(axis cs:28,992.69813989486)
--(axis cs:29,969.666702405361)
--(axis cs:29,1014.61641165734)
--(axis cs:29,1014.61641165734)
--(axis cs:28,1024.98045825907)
--(axis cs:27,1037.98009007565)
--(axis cs:26,1013.53510715325)
--(axis cs:25,1049.82264369189)
--(axis cs:24,1060.91319327565)
--(axis cs:23,1013.51797255503)
--(axis cs:22,1008.36358264406)
--(axis cs:21,1038.46676314726)
--(axis cs:20,1039.8248489073)
--(axis cs:19,1031.17978764957)
--(axis cs:18,1030.51390417376)
--(axis cs:17,1029.13284976411)
--(axis cs:16,1024.78981769866)
--(axis cs:15,1041.15529126541)
--(axis cs:14,1009.69179999668)
--(axis cs:13,1039.70633251273)
--(axis cs:12,987.567333281533)
--(axis cs:11,1010.55612302135)
--(axis cs:10,1026.9285267767)
--(axis cs:9,985.541307619796)
--(axis cs:8,921.608955654846)
--(axis cs:7,860.302951740679)
--(axis cs:6,945.713749365595)
--(axis cs:5,722.19233935038)
--(axis cs:4,565.26677564234)
--(axis cs:3,441.377565619606)
--(axis cs:2,300.198531484422)
--(axis cs:1,165.17222568846)
--(axis cs:0,52.9073221612647)
--cycle;

\path [draw=red, fill=red, opacity=0.2]
(axis cs:0,64.2186766120574)
--(axis cs:0,59.8232060837994)
--(axis cs:1,123.100333939021)
--(axis cs:2,163.370047282699)
--(axis cs:3,224.575870896228)
--(axis cs:4,271.031312527141)
--(axis cs:5,281.574499695604)
--(axis cs:6,347.143654716387)
--(axis cs:7,389.7371176947)
--(axis cs:8,432.709954678843)
--(axis cs:9,438.618490513994)
--(axis cs:10,497.205431722075)
--(axis cs:11,532.801202103185)
--(axis cs:12,521.365302824791)
--(axis cs:13,625.285589207059)
--(axis cs:14,619.796270963531)
--(axis cs:15,677.982020315766)
--(axis cs:16,692.6298359347)
--(axis cs:17,727.938886292242)
--(axis cs:18,746.298126052696)
--(axis cs:19,796.145328979575)
--(axis cs:20,783.403443936386)
--(axis cs:21,817.000108171288)
--(axis cs:22,819.817759032143)
--(axis cs:23,848.22680563992)
--(axis cs:24,821.193469930381)
--(axis cs:25,893.320940363289)
--(axis cs:26,862.877584789618)
--(axis cs:27,912.615791885327)
--(axis cs:28,911.702722146265)
--(axis cs:29,913.666292401596)
--(axis cs:29,948.673226533345)
--(axis cs:29,948.673226533345)
--(axis cs:28,942.609556135712)
--(axis cs:27,943.774180392609)
--(axis cs:26,901.949699768672)
--(axis cs:25,926.598595661527)
--(axis cs:24,880.045663840577)
--(axis cs:23,862.925026087151)
--(axis cs:22,845.821421021278)
--(axis cs:21,848.357714275512)
--(axis cs:20,816.014789611975)
--(axis cs:19,799.923061090217)
--(axis cs:18,772.670124034499)
--(axis cs:17,746.82152876137)
--(axis cs:16,721.091134478836)
--(axis cs:15,706.689694596812)
--(axis cs:14,666.214020469652)
--(axis cs:13,644.635595849473)
--(axis cs:12,608.956699620385)
--(axis cs:11,557.601862005179)
--(axis cs:10,511.246377881711)
--(axis cs:9,507.579964509402)
--(axis cs:8,448.166852527953)
--(axis cs:7,392.612094798223)
--(axis cs:6,374.697737639242)
--(axis cs:5,316.366239908685)
--(axis cs:4,291.717917140463)
--(axis cs:3,255.922710764474)
--(axis cs:2,195.488650644756)
--(axis cs:1,133.432719463485)
--(axis cs:0,64.2186766120574)
--cycle;

\addplot [semithick, green01270, dash pattern=on 5.55pt off 2.4pt]
table {%
0 177.928283987821
1 740.284224270483
2 1302.64016455314
3 1864.99610483581
4 2427.35204511847
5 2989.70798540113
6 3552.06392568379
7 4114.41986596645
8 4676.77580624911
9 5239.13174653177
10 5801.48768681444
11 6235.84854404574
12 6622.60578039599
13 6641.60511671547
14 6675.39918542521
15 6757.99561991331
16 6128.12094815419
17 6221.49082132613
18 6093.58217311641
19 6340.34319650332
20 6067.66526441397
21 6110.6180016251
22 7150.10498418414
23 6087.33585131182
24 6402.33561241631
25 6383.6910020534
26 6597.09002548831
27 5772.63806073965
28 6173.79488859897
29 6589.39235059562
};
\addlegendentry{\scriptsize $O\big(L\sqrt{T} \wedge BT \big)$ - Th. \ref{theorem:minimax_CT_param_free}}
\addplot [semithick, red, dash pattern=on 5.55pt off 2.4pt]
table {%
0 151.064948577729
1 573.29982612907
2 796.570626241297
3 969.730289586379
4 1116.34623056295
5 1245.82542053168
6 1363.06017856651
7 1470.98100679261
8 1571.50793644049
9 1665.97997174224
10 1755.37498506647
11 1840.43294095382
12 1921.72981506214
13 1999.72436951936
14 2074.78905578384
15 2147.2311701238
16 2217.3077751864
17 2285.23649582552
18 2351.20350327311
19 2415.36953399569
20 2477.87450402667
21 2538.84109972567
22 2598.3776095108
23 2656.58018391275
24 2713.53465898424
25 2769.31804194815
26 2823.99973254432
27 2877.64253536835
28 2930.30350532364
29 2982.03465862871
};
\addlegendentry{\scriptsize $O\big(\sqrt{LT} \wedge BT \big)$ - Cor. \ref{cor:regret_flat_pruning_exp_concave}}
\addplot [semithick, green01270]
table {%
0 51.6760951899452
1 161.132029365463
2 290.800974027772
3 415.663401067139
4 552.076773376039
5 649.792420990309
6 803.413480285618
7 848.800445422186
8 916.744559403817
9 890.743134123288
10 969.669363925463
11 989.570026405372
12 980.776262621567
13 1014.9523091584
14 977.125473064868
15 1009.49824658761
16 1012.70407952094
17 1013.59316772133
18 1024.14552926031
19 1009.09745484264
20 1010.81102161494
21 992.877673300602
22 981.427430602521
23 1000.05057742232
24 1006.85193737011
25 1042.00596025416
26 995.136066794684
27 1028.0265978411
28 1013.05964085976
29 993.702585654629
};
\addlegendentry{\scriptsize Chaining Tree - Alg. \ref{alg:training_CT}}
\addplot [semithick, red]
table {%
0 61.3647669776824
1 129.342975475272
2 173.474808596227
3 235.293878354083
4 275.016853625318
5 306.346670946183
6 357.421137397545
7 391.287981392269
8 441.731002762657
9 480.83593278763
10 501.081337549706
11 550.690332574532
12 564.97184835096
13 632.371669849314
14 641.426229155427
15 689.507410205711
16 697.801364785118
17 735.638849175355
18 766.54321847983
19 796.646296586495
20 789.940278791305
21 833.901568652401
22 836.032829390098
23 855.284035657029
24 864.999478303436
25 919.619090865231
26 876.076620477321
27 932.994353381976
28 933.928355428081
29 936.518843933974
};
\addlegendentry{\scriptsize Locally Adaptive Online Reg. - Alg. \ref{alg:Local_Adapt_Algo}}
\end{axis}

\end{tikzpicture}
        \label{fig:xp_lip_abs}
    }
  \caption{Regret (mean $\pm$ std over $5$ runs) as a function of $L$ for square and absolute loss functions, with a fixed horizon of $T=2000$. The analysis uses 20 equally spaced constants $l \in [2^{-6}, 2^5]$, which define the different Lipschitz functions where we apply our algorithms, given by $f_l(x) = f(lx)$ such that $\sup_{x \in \X} |f'_l(x)| \leq 15l =: L$.}
    \label{fig:xp_lip_combined}
\end{figure}
\vspace{-0.5cm}
\begin{figure}[htbp]
\centering
    \input{predictions}
    \caption{Predictions for Chaining Tree (Alg. \ref{alg:training_CT}) and Locally Adaptive Online Regression (Alg. \ref{alg:Local_Adapt_Algo}) after $T=1000$ data. For illustration purposes, we set the depth of the Chaining Trees to $5$ and that of the Core Tree to $3$. }
    \label{fig:xp_pred}
\end{figure}
\vspace{-0.5cm}
\emph{Note:} A minor adjustment has been made to the implementation of the Locally Adaptive Online Regression algorithm (Alg. \ref{alg:Local_Adapt_Algo}). Rather than performing a grid search to determine the root nodes of the CT in Core Tree \(\T_0\), we employ a \emph{Follow the Leader} (or best expert) strategy. For squared losses, this method offers a similar benefit, that is reducing the regret for learning the root nodes from \(O(B\sqrt{T})\) to \(O(B\log(T))\) - see \citet[Chap 3.2]{cesa2006prediction}. Consequently, the overall performance bound is improved, especially in low-dimensional cases (see Corollary \ref{cor:regret_flat_pruning_exp_concave}, exp-concave case).

\end{document}